\documentclass[a4paper,13pt,abstracton]{article}
 \usepackage[utf8]{inputenc}
   \providecommand{\keywords}[1]{\textbf{\textit{Key words:}} #1}
   \usepackage{amsmath,amsthm,verbatim,amssymb,amsfonts,amscd, graphicx,mathrsfs}
 \usepackage{geometry}
 \usepackage{tabularx}
 \usepackage{diagbox}
 \numberwithin{equation}{section}
 \newtheorem{thm}{Theorem}[section]
 \newtheorem{lem}[thm]{Lemma}
 \newtheorem{defi}[thm]{Definition}
 \newtheorem{cor}[thm]{Corollary}
 \newtheorem{prop}[thm]{Proposition}
 \newtheorem{rem}[thm]{Remark}

 \newtheorem{ex}[thm]{Example}
  
 \usepackage{enumerate}
 \usepackage[square,comma,super,numbers,sort&compress]{natbib}
 \usepackage{bm}
  \usepackage{hyperref}
 \hypersetup{colorlinks=true, urlcolor=Cerulean, citecolor=red}
 \usepackage[T1]{fontenc}
 \usepackage{nameref,cleveref}
 \usepackage{nicefrac,xfrac}
 \usepackage[all]{xy}
 \usepackage{color}
\usepackage[section]{placeins}
\usepackage{indentfirst}

\begin{document}
	\title{\textbf {Homogeneous ACM bundles on rational homogeneous spaces}}

	\author{Xinyi Fang
		\thanks{Department of Mathematics,
			Nanjing University,
			No. 22, Hankou Road,
			Nanjing, 210093, P. R. China,
			xyfang@nju.edu.cn.			
			The Research is Sponsored by Innovation Action Plan (Basic research projects) of Science and Technology Commission of Shanghai Municipality (Grant No. 21JC1401900) and Jiangsu Funding Program for Excellent Postdoctoral Talent.
		}
	}

\date{}
\maketitle


\begin{abstract}
In this paper, we characterize homogeneous arithmetically Cohen-Macaulay (ACM) bundles and Ulrich bundles on rational homogeneous spaces. 
From this result, we see that there are only finitely many irreducible homogeneous ACM bundles (up to twist) and Ulrich bundles on these varieties. Moreover, we give numerical criteria for some special irreducible homogeneous bundles to be ACM bundles.
\end{abstract}
\keywords{homogeneous ACM bundle, Ulrich bundle, rational homogeneous space}

\section{Introduction}
A well-known result of Horrocks (\cite{Horrocks1964}) states that a vector bundle $E$ on $\mathbb{P}^n$ splits as a direct sum of line bundles if and only if it has no intermediate cohomology, i.e. $h^i(\mathbb{P}^n, E(t))=0$ for $0<i<n$ and $t\in \mathbb{Z}$. It is thus natural to ask for the meaning of such a vanishing on other projective varieties. 
Let $(X, \mathcal{O}_X(1))$ be a polarized projective variety. A vector bundle $E$ on $X$ is called arithmetically Cohen-Macaulay (ACM) if all its intermediate cohomology modules $H^i(X, E(t))$ vanish for $0<i<\dim X$ and all $t\in \mathbb{Z}$. These sheaves  correspond to maximal Cohen-Macaulay modules on the coordinate ring. So ACM bundles provide a criterion to determine the complexity of the underlying variety.

The study of the
indecomposable ACM bundles on a given variety has attracted the attention of many mathematicians. Besides the result of Horrocks for projective spaces, Kn{\"o}rrer proved that indecomposable ACM bundles on a smooth quadric can only be  line bundles or twisted spinor bundles (\cite{Knorrer1987}). In fact, it is proved that varieties support only a finite number of indecomposable ACM bundles (up to twist) fall into a short list: three or less reduced points on $\mathbb{P}^2$, a rational normal curve, a projective space, a non-singular quadric hypersurface, a cubic scroll in $\mathbb{P}^4$ and the Veronese surface in $\mathbb{P}^5$ (\cite{Buchweitz1987, Eisenbud1988}). Recently, Faenzi and Pons-Llopis proved that most varieties support unbounded families of
indecomposable ACM bundles (\cite{F-P}). As a consequence, one should not expect to find a precise characterization of ACM bundles on other varieties. There are a lot of classical and recent papers devoted to classifying low-rank ACM bundles on particular varieties, like cubic surfaces, $K_3$ surfaces,  Calabi-Yau $3$-folds and hypersurfaces (\cite{Casanellas2011, Faenzi2008, Watanabe2008, Brambilla2011, Filip2014, Ravindra2019}). 


A widely studied class of ACM bundles is homogeneous ACM bundles on homogeneous varieties. The first result in this direction is the full characterization of irreducible homogeneous ACM bundles on Grassmannians (\cite{Costa2016}). Then Du, Ren, Nakayama and the author extended this result and classified irreducible homogeneous ACM bundles on rational homogeneous spaces with Picard rank $1$ (\cite{Du2023, Fang2023}). 
It is then quite natural to consider such a classification for a general rational homogeneous space. The problem gets more intricate due to the rich structure of the Picard group. 
This paper aims to classify all irreducible homogeneous ACM bundles on rational homogeneous spaces. Our motivation behind the result is to give a cohomological characterization of direct sums of line bundles on rational homogeneous spaces, as we note that homogeneous ACM bundles play an important role in the cohomological criteria (see \cite{ottaviani1989} Theorem 2.1 and Lemma 3.2).

To give a brief summary of our results we introduce some notation. Let $G$ be a simple Lie group and $P_I$ the parabolic subgroup of $G$ associated to the set $I$. Let $X=G/P_I$ be the corresponding rational homogeneous space. Let $L_{\varpi}$ be a very ample line bundle with weight $\varpi$ on $X$ giving an embedding $X\subset \mathbb{P}(V^{*})$, $V=H^0(X,L_{\varpi})$. We define the set $\Phi_{X}^{+}$ as follows:
$$\Phi_{X}^{+}:=\{\alpha\in\Phi^{+}_{G}\ |\ (\varpi,\alpha)\neq0\},$$
where $\Phi_{G}^{+}$ is the set of positive roots and $(,)$ denotes the Killing form. Denote by $E_{\lambda}$ the homogeneous bundle on $X=G/P_I$ arising from the irreducible representation of $P_I$ with highest weight $\lambda$. we define its {\it associated datum} $T^{\lambda}_{X}$ as follows:
$$T^{\lambda}_{X}:=\left\{\frac{(\lambda+\rho,\alpha)}{(\varpi,\alpha)} \ |\ \alpha\in\Phi_{X}^{+}\right\},$$
where $\rho$ is the sum of all fundamental weights. We set 
$$M^{\lambda}_{X}=\max\{t|t\in T^{\lambda}_{X}\}~\text{and}~m^{\lambda}_{X}=\min\{t|t\in T^{\lambda}_{X}\}.$$
 
\begin{thm}\label{thm} Let $E_\lambda$ be an irreducible homogeneous vector bundle with highest weight $\lambda$ over $X=G/P_I$. Let $T^{\lambda}_{X}$ be its associated datum.
	Denote $n_l:=\#\{t\in T^{\lambda}_{X}|t=l \}$. Then $E_\lambda$ is an ACM bundle if and only if $n_l\geq1$ for any integer $l$ between $m^{\lambda}_{X}$ and $M^{\lambda}_{X}$.
\end{thm}

This result was obtained in \cite{Costa2016, Du2023, Fang2023} for the case in which $G/P_I$ is a rational homogeneous space with Picard rank $1$. In this paper, we present a unified proof for all rational homogeneous spaces. From the main theorem, we can get the following corollary.

\begin{cor}\label{Cor} There are only finitely many irreducible homogeneous ACM bundles on $X$ up to twist by  $\mathcal{O}_{X}(1)$. 
\end{cor}

Among ACM bundles, there are bundles whose associated maximal Cohen-Macaulay modules have the maximal number of generators. Such bundles are called \emph{Ulrich} bundles which have been recently objects of deep inspection (\cite{Brennan1987, Casanellas2012, Lee2021}). Ulrich bundles have many good properties, thus their description is of particular interest. Costa and Mir{\'o}-Roig gave a full classification of all irreducible homogeneous Ulrich bundles on Grassmannians (\cite{Costa2015}). The results were then generalized to flag manifolds (\cite{Coskun2017}). At the same time, Fonarev classified irreducible homogeneous Ulrich bundles on isotropic Grassmannians (\cite{Fonarev2016}). Later, Lee and Park considered the same problem on exceptional homogeneous varieties (\cite{Lee2021-2}). In this paper, we give a complete classification of irreducible homogeneous Ulrich bundles on all rational homogeneous spaces.
\begin{thm}\label{thm1} Let $E_\lambda$ be an irreducible homogeneous vector bundle with highest weight $\lambda$ over $X=G/P_I$. Let $T^{\lambda}_{X}$ be its associated datum. Then $E_\lambda$ is an Ulrich bundle if and only if $T^{\lambda}_{X}=\{1,2,\ldots,\dim X\}$.
\end{thm}
\begin{cor}There are only finitely many irreducible homogeneous Ulrich bundles on $X$.
\end{cor}
 
\paragraph{Plan of the paper} In Section $2$, we fix the notation and collect basic facts about semi-simple Lie groups (algebras) and general results about homogeneous bundles on rational homogeneous spaces. In particular, we recall the Borel-Bott-Weil Theorem at the end of Section $2$. In Section $3$, we first show our main theorems on classifying the irreducible homogeneous ACM bundles and Ulrich bundles on all rational homogeneous spaces. Then, we explicitly describe the associated datum  with respect to the minimal ample class. In section $4$, we give simple numerical criteria to judge that some special irreducible homogeneous bundles are ACM bundles on $G/P$ with $G$ simple classical Lie group and we give the necessary and sufficient conditions under which line bundles are ACM bundles.

 \paragraph{Notation and convention}

\begin{itemize}
\item $X=G/P_I$: the rational homogeneous space with simple Lie group $G$ and parabolic subgroup $P_I$;	
\item  $\Phi_{G}^+$: the set of positive roots;
\item  $\Phi_{G}^-$: the set of negative roots;
\item  $\lambda_i$: the $i$-th fundamental weight;
\item $\rho$: $\lambda_1+\cdots+\lambda_n$;
\item $\alpha_i$: the $i$-th simple root;
\item  $(\cdot,\cdot)$: the Killing form;
\item  $L_{\varpi}$: the line bundle on $X$ with weight $\varpi$;
\item $E_\lambda$: the irreducible homogeneous vector bundle with highest weight $\lambda$;
\item $T^{\lambda}_{X}$: the associated datum of $E_\lambda$ on $X$.
\end{itemize}

\section{Preliminaries}
Throughout this paper, all algebraic varieties and morphisms will be defined over the field $\mathbb{C}$.

\subsection{Semi-simple Lie groups and algebras}
We fix the notation and recall some relevant concepts on semi-simple Lie groups and algebras. We refer to \cite{carter2005lie}, \cite{fulton2013representation} for further details.
\begin{enumerate}
	\item[(1)] Let $G$ be a semi-simple Lie group. The semi-simple Lie algebra $T_e{G}$ is called the Lie algebra associated to $G$ and it is denoted by $\mathfrak{g}$.
	
	\item[(2)] Let $\mathfrak{h}$ be a Cartan subalgebra of $\mathfrak{g}$ and $\Phi\subset \mathfrak{h}^{\ast}$ the roots of the pair $(\mathfrak{g},\mathfrak{h})$. By Cartan decomposition,
	\[\mathfrak{g}=\mathfrak{h}\oplus \sum_{\alpha\in  \Phi} \mathfrak{g}_\alpha,\]
	where $\mathfrak{g}_\alpha$ is a $1$-dimensional root space.
	Denote by $\Phi_{G}^{+}$ (resp. $\Phi_{G}^{-}$) $\subset\Phi$ the subset of positive (resp. negative) roots with respect to some ordering of the root system and by $\Delta$ a base for $\Phi_{G}^{+}$.
	
	\item[(3)] The Killing form $(\alpha, \beta):=tr(\text{ad}_\alpha\circ \text{ad}_\beta)$ defines a non-degenerate bilinear form on $\mathfrak{h}$, where $\alpha, \beta\in \mathfrak{g}$ and $\text{ad}$ is the adjoint representation of $\mathfrak{g}$. The Killing form can be extended to a non-degenerate bilinear form on $\mathfrak{h}^{\ast}$.
	
	\item[(4)] Denote by $\Lambda\subset \mathfrak{h}^{\ast}$ the weight lattice of $G$, given by
	\begin{align*}
		\Lambda=\{\lambda\in \mathfrak{h}^{\ast}|\frac{2(\lambda,\alpha)}{(\alpha,\alpha)}\in\mathbb{Z},\forall \alpha\in\Phi\}. 
	\end{align*}
	Let $n=|\Delta|$ be the rank of $\mathfrak{g}$. Index the collection of simple roots by $\alpha_{1},\ldots,\alpha_{n}$. Denote by $\lambda_{1},\ldots,\lambda_{n}$ the set of fundamental weights defined by $\frac{2(\lambda_i,\alpha_j)}{(\alpha_j,\alpha_j)}=\delta_{ij}$. The fundamental weights form a $\mathbb{Z}$-basis for $\Lambda$. A weight $\lambda=\sum_{i=1}^{n}a_{i}\lambda_{i}$ is said to be dominant if $a_{i}\geq 0$ for $1\leq i \leq n$ and strongly dominant if $a_{i}>0$. 
	
	\item[(5)] The Dynkin diagram of $G$, which we denoted by $\mathcal{D}:=\mathcal{D}(G)$, is determined by the Cartan matrix $(A_{ij}=2\frac{(\alpha_i,\alpha_j)}{(\alpha_i,\alpha_i)})$. It consists of a graph whose set of nodes is $\{1,2,\ldots,n\}$ and the nodes $i$ and $j$ are joined by $A_{ij}A_{ji}$ edges. When two nodes $i$ and $j$ are joined by a double or triple edge, we add to it an arrow pointing to $i$ if $|A_{ij}|>|A_{ji}|$. We call $\alpha_i$ a short root of $\mathcal{D}$ and $\alpha_j$ a  long root of $\mathcal{D}$. There is a one-to-one correspondence between isomorphism classes of semi-simple Lie algebras and Dynkin diagrams of reduced root systems. Moreover, every reduced root system is a disjoint union of mutually orthogonal irreducible root subsystems, each of them corresponding to one of the connected finite Dynkin diagrams $A_n$, $B_n$, $C_n$, $D_n$ $(n\in \mathbb{Z}_{>0})$, $E_6$, $E_7$, $E_8$, $F_4$, $G_2$. We use the Bourbaki labeling as follows.
	
	\setlength{\unitlength}{0.4mm}
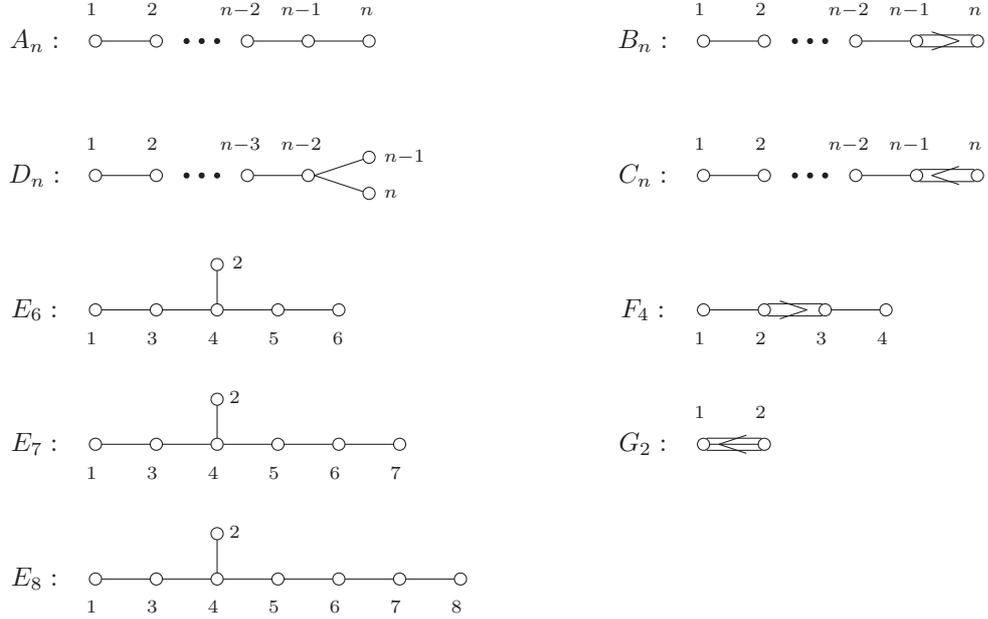
	\begin{center}
		\begin{picture}(280,0)(0,120)
			\put(10,100){\circle{4}} \put(30,100){\circle{4}}
			\put(60,100){\circle{4}} \put(80,100){\circle{4}}
			\put(12,100){\line(1,0){16}} \put(40,100){\circle*{2}}
			\put(45,100){\circle*{2}} \put(50,100){\circle*{2}}
			\put(62,100){\line(1,0){16}}
			\put(100,100){\circle{4}}
			\put(82,100){\line(1,0){16}}
			\put(-10,100){\makebox(0,0)[cc]{$A_n:$}}
			\put(7,110){$_1$}
			\put(27,110){$_2$}
			\put(51,110){$_{n-2}$}
			\put(71,110){$_{n-1}$}
			\put(97,110){$_{n}$}
			
			\put(210,100){\circle{4}} \put(230,100){\circle{4}}
			\put(260,100){\circle{4}} \put(280,100){\circle{4}}
			\put(212,100){\line(1,0){16}} \put(240,100){\circle*{2}}
			\put(245,100){\circle*{2}} \put(250,100){\circle*{2}}
			\put(262,100){\line(1,0){16}}
			\put(300,100){\circle{4}}
			\put(281,102){\line(1,0){18}}
			\put(281,98){\line(1,0){18}}
			\put(285,103){\line(3,-1){9}}
			\put(285,97){\line(3,1){9}}
			\put(190,100){\makebox(0,0)[cc]{$B_n:$}}
			\put(207,110){$_1$}
			\put(227,110){$_2$}
			\put(251,110){$_{n-2}$}
			\put(271,110){$_{n-1}$}
			\put(297,110){$_{n}$}
		\end{picture}
	\end{center}
	\vspace{.3cm}
	
	\begin{center}
		\begin{picture}(280,20)(0,120)
			\put(10,100){\circle{4}} \put(12,100){\line(1,0){16}}
			\put(30,100){\circle{4}} \put(40,100){\circle*{2}}
			\put(45,100){\circle*{2}} \put(50,100){\circle*{2}}
			\put(60,100){\circle{4}}
			\put(62,100){\line(1,0){16}} \put(80,100){\circle{4}}
			\put(82,100){\line(3,1){16}} \put(100,94){\circle{4}} \put(82,100){\line(3,-1){16}} \put(100,106){\circle{4}}
			\put(-10,100){\makebox(0,0)[cc]{$D_n:$}}
			\put(7,110){$_1$}
			\put(27,110){$_2$}
			\put(51,110){$_{n-3}$}
			\put(71,110){$_{n-2}$}
			\put(105,106){$_{n-1}$}
			\put(105,94){$_{n}$}
			
			\put(210,100){\circle{4}} \put(230,100){\circle{4}}
			\put(260,100){\circle{4}} \put(280,100){\circle{4}}
			\put(212,100){\line(1,0){16}} \put(240,100){\circle*{2}}
			\put(245,100){\circle*{2}} \put(250,100){\circle*{2}}
			\put(262,100){\line(1,0){16}}
			\put(300,100){\circle{4}}
			\put(281,102){\line(1,0){18}}
			\put(281,98){\line(1,0){18}}
			\put(285,100){\line(3,-1){9}}
			\put(285,100){\line(3,1){9}}
			\put(190,100){\makebox(0,0)[cc]{$C_n:$}}
			\put(207,110){$_1$}
			\put(227,110){$_2$}
			\put(251,110){$_{n-2}$}
			\put(271,110){$_{n-1}$}
			\put(297,110){$_{n}$}
		\end{picture}
	\end{center}
	\vspace{.3cm}
	
	\begin{center}
		\begin{picture}(280,20)(0,120)
			\put(10,100){\circle{4}} \put(12,100){\line(1,0){16}}
			\put(30,100){\circle{4}} \put(32,100){\line(1,0){16}}
			\put(50,100){\circle{4}} \put(52,100){\line(1,0){16}}
			\put(70,100){\circle{4}} \put(72,100){\line(1,0){16}}
			\put(90,100){\circle{4}} \put(50,102){\line(0,1){11}}
			\put(50,115){\circle{4}} \put(-10,100){\makebox(0,0)[cc]{$E_6:$}}
			\put(7,90){$_1$}
			\put(27,90){$_3$}
			\put(47,90){$_4$}
			\put(67,90){$_5$}
			\put(88,90){$_6$}
			\put(55,115){$_2$}
			
			\put(210,100){\circle{4}} \put(212,100){\line(1,0){16}}
			\put(230,100){\circle{4}} \put(231,102){\line(1,0){18}} \put(231,98){\line(1,0){18}}
			\put(250,100){\circle{4}} \put(252,100){\line(1,0){16}}
			\put(270,100){\circle{4}}
			\put(190,100){\makebox(0,0)[cc]{$F_4:$}}
			\put(235,103){\line(3,-1){9}}
			\put(235,97){\line(3,1){9}}
			\put(207,90){$_1$}
			\put(227,90){$_2$}
			\put(247,90){$_3$}
			\put(267,90){$_4$}
		\end{picture}
	\end{center}
	\vspace{.3cm}
	
	\begin{center}
		\begin{picture}(280,20)(0,120)
			\put(10,100){\circle{4}} \put(12,100){\line(1,0){16}}
			\put(30,100){\circle{4}} \put(32,100){\line(1,0){16}}
			\put(50,100){\circle{4}} \put(52,100){\line(1,0){16}}
			\put(70,100){\circle{4}} \put(72,100){\line(1,0){16}}
			\put(90,100){\circle{4}} \put(92,100){\line(1,0){16}}
			\put(110,100){\circle{4}} \put(50,102){\line(0,1){11}}
			\put(50,115){\circle{4}} \put(-10,100){\makebox(0,0)[cc]{$E_7:$}}
			\put(7,90){$_1$}
			\put(27,90){$_3$}
			\put(47,90){$_4$}
			\put(67,90){$_5$}
			\put(87,90){$_6$}
			\put(107,90){$_7$}
			\put(54,115){$_2$}
			
			\put(210,100){\circle{4}} \put(230,100){\circle{4}}
			\put(211,102){\line(1,0){18}}
			\put(212,100){\line(1,0){16}}
			\put(211,98){\line(1,0){18}}
			\put(190,100){\makebox(0,0)[cc]{$G_2:$}}
			\put(207,110){$_1$}
			\put(227,110){$_2$}
			\put(215,100){\line(3,-1){9}}
			\put(215,100){\line(3,1){9}}
		\end{picture}
	\end{center}
	\vspace{.3cm}
	
	\begin{center}
		\begin{picture}(280,20)(0,120)
			\put(10,100){\circle{4}} \put(12,100){\line(1,0){16}}
			\put(30,100){\circle{4}} \put(32,100){\line(1,0){16}}
			\put(50,100){\circle{4}} \put(52,100){\line(1,0){16}}
			\put(70,100){\circle{4}} \put(72,100){\line(1,0){16}}
			\put(90,100){\circle{4}} \put(92,100){\line(1,0){16}}
			\put(110,100){\circle{4}} \put(112,100){\line(1,0){16}}
			\put(130,100){\circle{4}} \put(50,102){\line(0,1){11}}
			\put(50,115){\circle{4}} \put(-10,100){\makebox(0,0)[cc]{$E_8:$}}
			\put(7,90){$_1$}
			\put(27,90){$_3$}
			\put(47,90){$_4$}
			\put(67,90){$_5$}
			\put(87,90){$_6$}
			\put(107,90){$_7$}
			\put(127,90){$_8$}
			\put(54,115){$_2$}
		\end{picture}
	\end{center}
	\vspace{2cm}
	
	The connected components of the Dynkin diagram $\mathcal{D}$ determine the simple algebraic groups that are factors of the semi-simple Lie group $G$, each of them corresponding to one of the Dynkin diagrams above.
\end{enumerate}

\subsection{Rational homogeneous space}
We review some well-known facts about rational homogeneous spaces and refer to \cite{ottaviani1995rational}, \cite{snow1989homogeneous} for more  details. Let $G$ be a semi-simple Lie group with associated Lie algebra $\mathfrak{g}$. Fix a Cartan subalgebra $\mathfrak{h}$ of $\mathfrak{g}$. We denote the algebra generated by $\mathfrak{h}$ and the negative root spaces by $\mathfrak{b}=\mathfrak{h}\oplus \sum_{\alpha\in \Phi_{G}^-}\mathfrak{g}_\alpha$. Let $B$ be the closed connected solvable subgroup of $G$ associated to $\mathfrak{b}$. We call $B$ and any of its conjugates in $G$ a Borel subgroup.

A parabolic subgroup $P$ of $G$ is a connected subgroup that contains a conjugate of $B$. The quotient $G/P$ is called a \emph{rational homogeneous space}. For simplicity, we usually assume that $P$ contains the Borel group $B$ just constructed. Since $G/B$ is compact and the projection $G/B\rightarrow G/P$ is surjective, $G/P$ is also compact, hence projective.  By the above arguments, the Lie algebra $\mathfrak{p}$ of $P$ decomposes into root spaces of $G$:
\begin{align}\label{para}
	\mathfrak{p}=\mathfrak{b}\oplus\sum_{\alpha\in\Phi^+_P}\mathfrak{g}_\alpha,  
\end{align}
where $\Phi^+_P$ is a subset of $\Phi_{G}^+$ that is closed under the addition of roots. Since $\Phi_{P}=\Phi_{G}^-\cup\Phi^+_P$ must be a closed set of roots and $\Phi_{P}$  contains all the negative simple roots, it follows that $\Phi^+_P$ is generated by a set of positive simple roots. Therefore, there is a one-to-one correspondence between the subsets $I\subset \{1,2,\ldots,n\}$ and parabolic subgroups $P$. In this paper, we denote 
$P_I$ by the parabolic subgroup of $G$, whose $\Phi^+_{P_I}$ is generated by simple roots in the complement of $I$. In particular, when $I=\{k\}$, $P_k$ is a maximal parabolic subgroup of $G$. For example, $A_n/P_k$ is the usual Grassmannian $G(k,n+1)$.

Since every semi-simple Lie group can be decomposed into a direct product of simple Lie groups, every rational homogeneous space $G/P$ can be decomposed into a product
\[G/P\simeq G_1/P_{I_1}\times G_2/P_{I_2}\times \cdots \times G_m/P_{I_m}\]
of rational homogeneous spaces with simple Lie groups $G_1, \ldots, G_m$. Hence here we only consider the case $G$ is a simple Lie group. Below we list some examples of rational homogeneous spaces and give their geometric explanations.
\begin{ex}
	\begin{enumerate}
		\item[(1)] $A_{n-1}/P_{d_1,\ldots,d_s}$ is the usual \emph{flag manifold} parametrizing flags
		$$V_{d_1}\subset V_{d_2}\subset\cdots\subset V_{d_s}\subset \mathbb{C},$$
		where $\dim(V_{d_i})=d_i~(1\le i\le s)$.
		\item[(2)] Let $V=\mathbb{C}^{2n+1}$ (resp. $\mathbb{C}^{2n}$) be a vector space equipped with a non-degenerate symmetric bilinear form $\mathcal{Q}$. Then $B_n/P_{d_1,\ldots,d_s}$ (resp. $D_n/P_{d_1,\ldots,d_s}~(d_s\le n-2)$) is the \emph{odd (resp. even) orthogonal flag manifold}, which parametrizes flags
		$$V_{d_1}\subset V_{d_2}\subset\cdots\subset V_{d_s}\subset V,$$
		where each $V_{d_i}~(1\le i\le s)$ is a $d_i$-dimensional isotropic subspace in $V$.
		\item[(3)] Let $\bar{V}=\mathbb{C}^{2n}$ be a vector space equipped with a non-degenerate skew-symmetric bilinear form $\Omega$. Then $C_n/P_I$ is the \emph{symplectic flag manifold}, which parametrizes flags
		$$V_{d_1}\subset V_{d_2}\subset\cdots\subset V_{d_s}\subset \bar{V},$$
		where each $V_{d_i}~(1\le i\le s)$ is a $d_i$-dimensional isotropic subspace in $\bar{V}$.
	\end{enumerate}
\end{ex}
\subsection{Homogeneous  vector bundles}
We now concentrate on the case of homogeneous vector bundles on $G/P$.
It's natural to define a homogeneous vector bundle $E$ if the action of $G$ on $G/P$ can be lifted to $E$. 

\begin{defi}
	Let $E$ be a vector bundle on  $G/P$. $E$ is called \emph{homogeneous} if there exists an action $G$ over $E$ such that the following diagram commutes
	
\centerline{
	\xymatrix{   G\times E \ar[r]\ar[d]& E\ar[d]\\
		G\times G/P \ar[r] & G/P. }}		
\end{defi}

It's evident from this definition that the tangent bundles $T(G/P)$ are homogeneous. A well-known fact is that all line bundles on $G/P$ are homogeneous (see \cite{snow1989homogeneous} Theorem 6.4). This is certainly not the case for vector bundles of higher rank. We will use in this paper another equivalent definition of homogeneous bundles that is useful and easier to handle.
\begin{thm}(\cite{ottaviani1995rational} Theorem 9.7)
	A vector bundle $E$ of rank $r$ on $G/P$ is homogeneous if and only if there exists a representation $\upsilon:P\to GL(r)$  such that $E\cong E_{\upsilon}$.
\end{thm}

If a representation $\upsilon:P\to GL(V)$ is irreducible, then we call the corresponding bundle $E_{\upsilon}$ an \emph{irreducible homogeneous vector bundle}. By Ise's result, 
any irreducible representation of $P$ is the tensor product of a $1$-dimensional representation of a torus and an irreducible representation of the semi-simple Lie group $S_P$, which is the semi-simple part of $P$. Hence irreducible homogeneous vector bundles on $G/P_I$ have been completely classified (c.f. \cite{ottaviani1995rational} Proposition 10.9). If $\lambda'$ is the highest weight of $V$ as an irreducible $S_{P_I}$-module ($\lambda'$ is dominant by highest weight theory), we define $\lambda'+\sum_{i\in I}n_i\lambda_i~(n_i\in \mathbb{Z})$ to be the highest weight of the irreducible $P_I$-module $V\bigotimes \otimes_{i\in I}L^{n_i}_{\lambda_i}$, where $L_{\lambda_i}$ is the one-dimensional representation defined by the $i$-th fundamental weight $\lambda_i$.

\begin{ex}\label{ex}
	\begin{enumerate}
		\item[(1)] On $X=B_n/P_{d_1,\ldots,d_s}~(d_s\le n-1)$ (resp. $D_n/P_{d_1,\ldots,d_s}~(d_s\le n-2)$), there exists a flag of universal subbundles
		$$0=H_0\subset H_{d_1}\subset \cdots\subset H_{d_s}\subset H_{d_s}^{\bot}\subset\cdots\subset H_{d_1}^{\bot}\subset\mathcal{O}_{X}\times V,$$ 
		where $H_{d_i}$ is a rank $d_i$ vector bundle whose fiber at a flag (point) $V_{d_1}\subset \cdots V_{d_i}\cdots\subset V_{d_s}$ is the subspace $V_{d_i}$ and $H_{d_i}^{\bot}$ is the \emph{$\mathcal{Q}$-orthogonal complement} of $H_{d_i}$, whose rank is $2n+1-d_i$ (resp. $2n-d_i$).
		
		\item[(2)] There is a flag of universal subbundles
		$$0=H_0\subset H_{d_1}\subset \cdots\subset H_{d_s}\subset H_{d_s}^{\bot}\subset\cdots\subset H_{d_1}^{\bot}\subset\mathcal{O}_{X}\times\bar{V}$$
		on $C_n/P_{d_1,\ldots,d_s}$, where $H_{d_i}^{\bot}$ is the \emph{$\Omega$-orthogonal complement} of $H_{d_i}$ with rank $2n-d_i$.
	
There are three important classes of irreducible homogeneous bundles on
$G/P_{d_1,\ldots,d_s}$, where $G$ is a simple Lie group of type $B_n$, $C_n$ or $D_n$.	
		\begin{itemize}
			\item $H_{d_i}/H_{d_{i-1}}$ is an irreducible homogeneous bundle with highest weight $\lambda_{d_i-1}-\lambda_{d_i}~(1\le i\le s)$.
			\item $(H_{d_i}/H_{d_{i-1}})^{\ast}$ is an irreducible homogeneous bundle with highest weight $\lambda_{d_{i-1}+1}-\lambda_{d_{i-1}}~(1\le i\le s)$.
			\item $H_{d_s}^{\bot}/H_{d_s}$ is an irreducible homogeneous bundle with highest weight 	\begin{align*}
				\left\{
				\begin{array}{lll}
					2\lambda_n-\lambda_{n-1}, & \text{if}~X=B_n/P_{d_1,\ldots,d_s}~\text{and}~d_s=n-1;\\
					\lambda_n+\lambda_{n-1}-\lambda_{n-2}, & \text{if}~X=D_n/P_{d_1,\ldots,d_s}~\text{and}~d_s=n-2;\\
					\lambda_{d_s+1}-\lambda_{d_s}, & \text{otherwise}.\\
				\end{array}
				\right.   
			\end{align*}
		\end{itemize}      
	\end{enumerate}  
\end{ex}

In this paper, we denote $E_\lambda$ by the irreducible homogeneous bundle on $G/P_I$ with highest weight $\lambda$. By the above arguments, $\lambda=\sum a_i\lambda_i$, where $a_i\geq0$ for $i\notin I$. We shall see that irreducible homogeneous bundles often have the sharpest theorems and the most detailed formulas associated to them. Unfortunately, many interesting bundles are not irreducible, for example, most tangent bundles on $G/P$. Nevertheless, it is often possible to draw conclusions about an arbitrary homogeneous vector bundle from the irreducible case by considering a filtration (c.f. \cite{snow1989homogeneous} Section $5$). Hence we only consider the irreducible homogeneous vector bundles.

A remarkable result is that $Pic(G/P_I)$ is freely generated by the classes of $L_{\lambda_i}~(i\in I)$, where $L_{\lambda_i}$ is the line bundle with weight $\lambda_i$. The next theorem addresses the question of when a line bundle is spanned or is ample.

\begin{thm}[Borel-Weil, see \cite{snow1989homogeneous} Theorem 6.5]\label{bb}
	Let $L$ be a line bundle on $X=G/P_I$ with weight $\lambda=\sum_{i\in I}n_i\lambda_i$. Then 
	\begin{enumerate}
		\item[(1)] $L$ is spanned at one point of $X$ iff $L$ is spanned at every point of $X$ iff $n_i\ge 0$ for any $i\in I$ (i.e. iff $\lambda$ is dominant).
		\item[(2)] $L$ is ample iff $L$ is very ample iff $n_i>0$ for any $i\in I$ (i.e. iff $\lambda$ is strongly dominant).
		\item[(3)] If $\lambda$ is dominant, then $H^0(X,L)$ is isomorphic to the irreducible $G$-module $V^{\lambda}$.
		\item[(4)] If $\lambda$ is not dominant, then $H^0(X,L)=0$. 
	\end{enumerate}
\end{thm}

Throughout this paper, we call the line bundle $L_X=\otimes_{i\in I}L_{\lambda_i}$ the \emph{minimal ample line bundle} on $X$. From the above theorem, we know that it is very ample. The following theorem generalizes the theorem of Borel-Weil, in the sense that it describes all cohomology groups of any irreducible homogeneous bundle.
\begin{thm}[Borel-Bott-Weil, see \cite{ottaviani1995rational} Theorem 11.4]\label{bbw} Let $E_\lambda$ be an irreducible homogeneous vector bundle on $G/P.$
	\begin{enumerate}
		\item[(1)] If $\lambda+\rho$ is singular, then $$H^i(G/P,E_\lambda)=0, \forall i\in\mathbb{Z}.$$
		\item[(2)] If $\lambda+\rho$ is regular of index p, then $$H^i(G/P,E_\lambda)=0, \forall i\neq p,$$
		and $$H^p(G/P,E_\lambda)=G_{w(\lambda+\rho)-\rho},$$ where $\rho=\sum\limits_{i=1}^n\lambda_i$ and $w(\lambda+\rho)$ is the unique element of the fundamental Weyl chamber of G which is congruent to $\lambda+\rho$ under the action of the Weyl group.
	\end{enumerate}
\end{thm}

\section{Homogeneous ACM bundles on rational homogeneous spaces}
The aim of this paper is to classify irreducible homogeneous ACM vector bundles on rational homogeneous spaces. Let's recall here the definitions:

\begin{defi}
	Let $X\subset \mathbb{P}^N$ be a projective variety and $\mathcal{O}_X(1)$ the corresponding very ample line bundle. We say $E$ is an arithmetically Cohen-Macaulay (ACM) sheaf on $X$ if $h^i(X,E\otimes\mathcal{O}_X(t))=0$ for all $i~(0<i<\dim X)$ and $t\in\mathbb{Z}$.  
\end{defi}

When $X$ is non-singular, ACM sheaves are locally free. For this reason we will speak ACM bundles in this case. Recall a subscheme $X\subseteq\mathbb{P}^N$ is arithmetically Cohen-Macaulay (ACM for short) if its homogeneous coordinate ring $R_X$ is a local Cohen-Macaulay ring.
There is a one-to-one correspondence between ACM sheaves on ACM scheme $X$ and graded maximal Cohen-Macaulay $R_X$-modules sending $E$ to $H_{*}^0(E):=\oplus_{t\in \mathbb{Z}}H^0_{*}(X,E(t))$ (see \cite{Casanellas20005} Proposition 2.1). 

In this paper, we consider ACM bundles on $X=G/P_I\subset\mathbb{P}V$ with respect to general polarizations, i.e. the corresponding line bundles $\mathcal{O}_{X}(1)$ are very ample line bundles $L_{\varpi}$, whose weight are $\varpi=\sum_{i\in I}n_i\lambda_i,~n_i>0$.


\subsection{Proof of Theorem \ref{thm} and Theorem \ref{thm1}}
In this section, we give a combinatorial criterion for an irreducible homogeneous vector bundle E on a rational homogeneous space $G/P_I$ to be ACM, formulated in terms of the weight of the representation attached to $E$. From the Borel-Bott-Weil theorem, we first get the following lemma.

\begin{lem}\label{lemma1}
	Let $E_{\lambda}$ be an irreducible homogeneous vector bundle on $G/P_I$ with highest weight $\lambda$. Then $E_{\lambda}$ is an ACM bundle if and only if one of the following conditions holds for each $t\in \mathbb{Z}$:
	\begin{enumerate}
		\item[(1)] $\lambda+\rho-t\varpi$ is regular of index $0$;
		\item[(2)] $\lambda+\rho-t\varpi$ is regular of index $\dim G/P_I$;
		\item[(3)] $\lambda+\rho-t\varpi$ is singular.
	\end{enumerate}
\end{lem}
\begin{proof}
	By definition, $E_{\lambda}$ is an ACM bundle is equivalent to say that any twist $E_{\lambda}(-t)$ of $E_{\lambda}$ has at most cohomology concentrated in zero or highest degree. Notice that the highest weight of $E_{\lambda}(-t)$ is $\lambda-t\varpi$.
	Hence we complete the proof immediately by the Borel-Bott-Weil theorem.
\end{proof}

Next, We examine the three conditions in Lemma \ref{lemma1} one by one. Recall that a weight $\omega$ is called singular if there exists $\alpha \in \Phi_{G}^{+}$ such that $(\omega,\alpha)=0$. $\omega$ is regular of index $p$ if $\omega$ is not singular and
there are exactly $p$ roots $\alpha_{1},\ldots,\alpha_{p} \in \Phi_{G}^{+}$ such that $(\omega,\alpha_{i})<0$. Let 
$$\Phi_{G/P_I}^{+}:=\{\alpha\in\Phi^{+}_{G}\ |\ (\varpi,\alpha)\neq0\}.$$
For an irreducible homogeneous bundle $E_{\lambda}$ on $G/P_I$. we define its associated datum
$$T^{\lambda}_{G/P_I}:=\left\{\frac{(\lambda+\rho,\alpha)}{(\varpi,\alpha)} \ |\ \alpha\in\Phi_{G/P_I}^{+}\right\}.$$
And we set
$$M^{\lambda}_{G/P_I}=\max\{t|t\in T^{\lambda}_{G/P_I}\}~\text{and}~m^{\lambda}_{G/P_I}=\min\{t|t\in T^{\lambda}_{G/P_I}\}.$$

\begin{lem}\label{m}
	$\lambda+\rho-t\varpi$ is regular of index $0$ if and only if $t<m^{\lambda}_{G/P_I}$. 
\end{lem}
\begin{proof}
	First of all, from the definition of $\Phi_{G/P_I}^{+}$, we see that for any positive root $\alpha\in \Phi^{+}_{G}\backslash\Phi_{G/P_I}^{+}$,
	$$(\lambda+\rho-t\varpi,\alpha)=(\lambda+\rho,\alpha)>0.$$
	Hence $\lambda+\rho-t\varpi$ is regular of index $0$ if and only if $(\lambda+\rho-t\varpi,\alpha)>0$ for any positive root $\alpha\in \Phi_{G/P_I}^{+}$, which means that 
	$$t<\frac{(\lambda+\rho,\alpha)}{(\varpi,\alpha)}~\text{for any}~\alpha\in \Phi_{G/P_I}^{+},$$
	i.e., $t<m^{\lambda}_{G/P_I}$. 
\end{proof}

\begin{rem}\label{min}
	Suppose $\varpi=\sum_{i\in I}n_i\lambda_i~(n_i>0)$ and 
	 $\lambda=\sum_{i=1}^{n}a_i\lambda_i$, where $n$ is the rank of $G$. Then we claim that 
	\begin{align}
		m^{\lambda}_{G/P_I}=\min_{i\in I}\{\frac{a_i+1}{n_i}\}.   
	\end{align}
	\begin{proof} First of all, when $\alpha=\alpha_k~(k\in I)$, we have
		$$\frac{(\lambda+\rho,\alpha)}{(\varpi,\alpha)}=\frac{(a_k+1)(\lambda_k,\alpha_k)}{n_k(\lambda_k,\alpha_k)}=\frac{a_k+1}{n_k}.$$
		On the other hand, for any $\alpha\in\Phi_{G/P_I}^{+}$, we expand $\frac{(\lambda+\rho,\alpha)}{(\varpi,\alpha)}$ as follows.
		\begin{align*}
			\frac{(\lambda+\rho,\alpha)}{(\varpi,\alpha)}
			&= \frac{(\sum_{i=1}^{n}(a_i+1)\lambda_i,\alpha)}{(\varpi,\alpha)} \\
			&=\frac{(\sum_{i\in I}(a_i+1)\lambda_i,\alpha)+(\sum_{i\notin I}(a_i+1)\lambda_i,\alpha)}{(\sum_{i\in I}n_i\lambda_i,\alpha)}\\
			&\ge\min_{i\in I}\{\frac{a_i+1}{n_i}\}.\\
			&(\text{since}~(\lambda_i,\alpha)\ge 0~\text{and}~a_i\ge 0 ~\text{for all}~i\notin I)      
		\end{align*}
		Thus, we get that $m^{\lambda}_{G/P_I}=\min_{i\in I}\{\frac{a_i+1}{n_i}\}$.
	\end{proof}
\end{rem}

\begin{lem}\label{M}
	$\lambda+\rho-t\varpi$ is regular of index $\dim G/P_I$ if and only if $t>M^{\lambda}_{G/P_I}$. 
\end{lem}
\begin{proof}
	According to Equation (\ref{para}), we have 
	$$\dim G/P_I=\dim \mathfrak{g}/\mathfrak{p}_I=|\Phi^{+}_{G}\backslash\Phi^{+}_{P_I}|,$$
	where $\Phi^{+}_{P_I}$ is the subset of $\Phi^{+}_{G}$ generated by $\alpha_i~(i\notin I)$. It's not hard to see that $\Phi^{+}_{P_I}$ is the complement of $\Phi_{G/P_I}^{+}$ in $\Phi^{+}_{G}$. Hence $\dim G/P_I=|\Phi_{G/P_I}^{+}|$.
	
	Since $(\lambda+\rho-t\varpi,\alpha)>0$ for any $\alpha\in \Phi^{+}_{G}\backslash\Phi_{G/P_I}^{+}$, $\lambda+\rho-t\varpi$ is regular of index $\dim G/P_I=|\Phi_{G/P_I}^{+}|$ is equivalent to saying that $(\lambda+\rho-t\varpi,\alpha)<0$ for any $\alpha\in \Phi_{G/P_I}^{+}$, which means that 
	$$t>\frac{(\lambda+\rho,\alpha)}{(\varpi,\alpha)}~\text{for any}~\alpha\in \Phi_{G/P_I}^{+},$$
	i.e., $t>M^{\lambda}_{G/P_I}$. 
\end{proof}

\begin{lem}\label{L}
	$\lambda+\rho-t\varpi$ is singular if and only if $t\in T^{\lambda}_{G/P_I}$. 
\end{lem}
\begin{proof}
	If $\lambda+\rho-t\sum_{i\in I}\lambda_i$ is singular, there exists a positive root $\alpha$ in $\Phi_{G/P_I}^{+}$ such that
	$$0=(\lambda+\rho-t\varpi,\alpha)=(\lambda+\rho,\alpha)-t.(\varpi,\alpha).$$
	Since $\alpha\in\Phi_{G/P_I}^{+}$, $(\varpi,\alpha)$ is not equal to zero. Thus, $t=\frac{(\lambda+\rho,\alpha)}{(\varpi,\alpha)}$. By the definition of $T^{\lambda}_{G/P_I}$, $t\in T^{\lambda}_{G/P_I}$. Conversely, if $t\in T^{\lambda}_{G/P_I}$, there exists a positive root $\alpha\in\Phi_{G/P_I}^{+}$ such that $t=\frac{(\lambda+\rho,\alpha)}{(\varpi,\alpha)}$. Then,
	$$0=(\lambda+\rho,\alpha)-t.(\varpi,\alpha)=(\lambda+\rho-t\varpi,\alpha).$$
	Therefore, $\lambda+\rho-t\varpi$ is singular.    
\end{proof}

By the above lemmas, we conclude Theorem \ref{thm}
immediately. As a consequence of Theorem \ref{thm}, we obtain the following corollary.
\begin{cor}\label{c}
 There are only finitely many irreducible homogeneous ACM bundles on $X=G/P_I$ up to twist by $\mathcal{O}_{X}(1)$.
\end{cor}
\begin{proof}
	Let $E_{\lambda}$ be an irreducible homogeneous vector bundle with $\lambda=\sum_{i=1}^{n}a_{i}\lambda_{i}$. After twisting with an appropriate power of $\mathcal{O}_{X}(1)$, we can assume that $a_i\ge 0$ for any $i\in I$ and $0<m^{\lambda}_{G/P_I}\le 1$ by Remark \ref{min}. Since $\dim G/P_I=|\Phi_{G/P_I}^{+}|$, $T^{\lambda}_{G/P_I}$ has at most $\dim G/P_I$ different elements. Thus if $M^{\lambda}_{G/P_I}\ge\dim G/P_I+1$, there exists an integer $l \in[m^{\lambda}_{G/P_I},M^{\lambda}_{G/P_I}]$ such that $n_{l}=0$. By the main theorem, $E_{\lambda}$ is not an ACM bundle. Therefore, if $E_{\lambda}$ is an ACM bundle, then $M^{\lambda}_{G/P_I}<\dim G/P_I+1$. Take $\alpha=\sum_{i=1}^{n}\alpha_{i}$, where $\alpha_{i}$ is the $i$-th simple root of $G$. It's obvious that $\alpha\in\Phi_{G/P_I}^{+}$ and $(\lambda_i,\alpha)>0$ for any $1\le i\le n$. It follows that $\frac{(\lambda+\rho,\alpha)}{(\varpi,\alpha)}$ is a linear combination of $a_{i}$ with positive rational coefficients. Since  $\frac{(\lambda+\rho,\alpha)}{(\varpi,\alpha)}$ is less than $\dim G/P_I+1$ and $a_{i}\geq0$, 
	there exist only a finite number of choices for $a_{i}$. Therefore, there exist only a finite number of irreducible homogeneous ACM bundles.
\end{proof}

Among ACM bundles, there are bundles $E$ such that $H^0_{*}(X,E)$ has the highest possible number of generators in degree $0$. Such bundles are called Ulrich bundles. There are several equivalent definitions of Ulrich bundles using linear resolutions or cohomologies. We are going to use the original cohomology characterization.
\begin{defi}
Let $X\subset \mathbb{P}^N$ be a projective variety and $\mathcal{O}_X(1)$ the corresponding very ample line bundle. We say $E$ is an Ulrich bundle on $X$ if
\begin{align*}
	\left\{
	\begin{array}{lll}
		h^i(X,E\otimes\mathcal{O}_X(-t))=0, & \text{for all}~\le i \le\dim X~\text{and all}~t\in \mathbb{Z};\\
		h^0(X,E\otimes\mathcal{O}_X(-t))=0, & \text{for}~t\ge 1;\\
		h^{\dim X}(X,E\otimes\mathcal{O}_X(-t))=0, & \text{for}~t\le \dim X.\\
	\end{array}
	\right.   
\end{align*}
\end{defi}

It's worth noting that the property of being Ulrich is not maintained after tensoring with line bundles. Using the associated datum of irreducible homogeneous vector bundles, we come to prove Theorem \ref{thm1}.

\textbf{Proof of Theorem \ref{thm1}:}
Let $E_{\lambda}$ be an irreducible homogeneous Ulrich bundle on $X=G/P_I$.
First we claim that $m^{\lambda}_{G/P_I}\le 1$ and $M^{\lambda}_{G/P_I}\ge \dim X$. This is because if $m^{\lambda}_{G/P_I}>1$, then $h^0(X,E_{\lambda}\otimes\mathcal{O}_X(-1))$ would not be equal to $0$ by Lemma \ref{m} and the Borel-Bott-Weil theorem, contrary to the definition of Ulrich bundles. Hence $m^{\lambda}_{G/P_I}\le 1$. Similarly, from Lemma \ref{M}, we can deduce that $M^{\lambda}_{G/P_I}\ge \dim X$. On the other hand, since $E_{\lambda}$ is an ACM bundle and $T^{\lambda}_{G/P_I}$ has at most $\dim X$ different elements, we have 
\[
0<m^{\lambda}_{G/P_I}\le 1~\text{and}~\dim X\le M^{\lambda}_{G/P_I}<\dim X+1
\] 
by Theorem \ref{thm}. Further, since all integers between $1$ and $\dim X$ appear as entries of $T^{\lambda}_{G/P_I}$, which has at most $\dim X$ different entries, we have $T^{\lambda}_{X}=\{1,2,\ldots,\dim X\}$.

 Conversely, if $T^{\lambda}_{X}=\{1,2,\ldots,\dim X\}$, then $m^{\lambda}_{G/P_I}=1$ and $ M^{\lambda}_{G/P_I}=\dim X$. By Theorem \ref{thm}, $E_{\lambda}$ is an ACM bundle. Moreover, since $m^{\lambda}_{G/P_I}=1$, according to the Borel-Bott-Weil theorem and Lemma \ref{m}, we have $H^0(X, E_{\lambda}(-t))=0$ for any $t\ge 1$. Similarly, since $M^{\lambda}_{G/P_I}=\dim X$, we have $H^{\dim X}(X, E_{\lambda}(-t))=0$ for any $t\le \dim X$ due to Lemma \ref{M}. Therefore, $E_{\lambda}$ is an Ulrich bundle.

\begin{cor}There are only finitely many irreducible homogeneous Ulrich bundles on $X=G/P_I$.
\end{cor}
\begin{proof}
From Theorem \ref{thm1}, we see that if $E_{\lambda}$ is an irreducible homogeneous Ulrich bundle on $X$, then $m^{\lambda}_{G/P_I}=1$ and $ M^{\lambda}_{G/P_I}=\dim X$. Hence we conclude the proof by applying almost verbatim the proof of Corollary \ref{c}.
\end{proof}
\subsection{The associated datum of ACM bundles on $G/P_I$}
From now on, for convenience, we always suppose that $X=G/P_I\subset\mathbb{P}V$ in its \emph{minimal embedding}, i.e. the corresponding line bundle $\mathcal{O}_{X}(1)$ is the minimal ample line bundle $L_X$, whose weight is $\sum_{i\in I}\lambda_i$.
\subsubsection{The associated datum for classical type}
In this section, we give a concrete description of the associated datum $T^{\lambda}_{G/P_I}$ on $G/P_I$, where $G$ is a classical Lie group of type $A$, $B$, $C$ or $D$. From now, we write $\lambda=a_1\lambda_1+\cdots+a_n\lambda_n$ and 
$I=\{d_1,\ldots,d_s\}\subset\{1,\ldots,n\}$. (By convention we set $d_0=0$ and $d_{s+1}=n$.)

Let us start with the positive roots and fundamental weights of Lie groups of types $B_n$, $C_n$ and $D_n$, which can be found in the appendix of \cite{carter2005lie}.

\begin{lem} \label{fund} 
	We define \[\mathfrak{e}=\left\{\begin{matrix}\frac{1}{2}& \text{if G is of type B},\\
		1& \text{if G is of type C},\\
		0& \text{if G is of type D}.\\
	\end{matrix}\right.\]
	For Lie algebras of types $B_n$, $C_n$ and $D_n$, we can take orthogonal bases of the $\mathbb{R}$-vector space spanned by the vectors corresponding to the simple roots of these Lie algebras such that the positive roots are $$\Phi^+_B =\{e_i+e_j\}_{i<j}\cup\{e_i-e_j\}_{i<j}\cup\{e_i\}_i,$$
	$$\Phi^+_C =\{e_i+e_j\}_{i\leq j}\cup\{e_i-e_j\}_{i<j},$$
	$$\Phi^+_D =\{e_i+e_j\}_{i<j}\cup\{e_i-e_j\}_{i<j}.$$
	The fundamental weights are $$\lambda_i^{B,C,D}=e_1+\dots+e_i~\text{for}~i\leq n-2+2\mathfrak{e}$$ and
	$$\lambda_n^{B,C}=\mathfrak{e}(e_1+\dots+e_{n-1}+e_n),$$
	$$\lambda_{n-1}^D=\frac{1}{2}(e_1+\dots+e_{n-1}-e_n),$$
	$$\lambda_n^{D}=\frac{1}{2}(e_1+\dots+e_{n-1}+e_n).$$
	$$So ~\rho=(n+\mathfrak{e}-1)e_1+(n+\mathfrak{e}-2)e_2+\dots+\mathfrak{e}e_n.$$
\end{lem}

\textbf{I. The associated datum on $A_{n-1}/P_I$}

The associated datum $T^{\lambda}_{A_{n-1}/P_I}$ on $A_{n-1}/P_I=A_{n-1}/P_{d_1,\ldots,d_s}$
consists of ${s+1\choose 2}$ matrices, which has already appeared in \cite{Fang2022}, we repeat it here for the sake of self-consistency.

$$T^{\lambda}_{A_{n-1}/P_I}=\{(A^{ij})|1\le i\le j\le s\},$$
where $(A^{ij})$ is a $((d_i-d_{i-1})\times (d_{j+1}-d_j))$ matrix with coefficients 
\begin{align}\label{matrix1}
	A^{ij}_{uv}=\frac{\sum\limits_{k=d_i-u+1}^{d_j+v-1}(a_k+1)}{j-i+1}.
\end{align}
\begin{equation*}\label{matrix}
	\left(
	\begin{array}{ccccc}
		\frac{\sum\limits_{k=d_i}^{d_j}a_k+d_j-d_i+1}{j-i+1}&\cdots&	\frac{\sum\limits_{k=d_i}^{d_j+v-1}a_k+d_j-d_i+v}{j-i+1}&\cdots&	\frac{\sum\limits_{k=d_i}^{d_{j+1}-1}a_k+d_{j+1}-d_i}{j-i+1}\\
		\vdots&\ddots&\vdots&\ddots&\vdots\\
		\frac{\sum\limits_{k=d_i-u+1}^{d_j}a_k+d_j-d_i+u}{j-i+1}&\cdots&	\frac{\sum\limits_{k=d_i-u+1}^{d_j+v-1}a_k+d_j-d_i+u+v-1}{j-i+1}&\cdots&	\frac{\sum\limits_{k=d_i-u+1}^{d_{j+1}-1}a_k+d_{j+1}-d_i+u-1}{j-i+1}\\
		\vdots&\ddots&\vdots&\ddots&\vdots\\
		\frac{\sum\limits_{k=d_{i-1}+1}^{d_j}a_k+d_j-d_{i-1}}{j-i+1}&\cdots&	\frac{\sum\limits_{k=d_{i-1}+1}^{d_j+v-1}a_k+d_j-d_{i-1}+v-1}{j-i+1}&\cdots&	\frac{\sum\limits_{k=d_{i-1}+1}^{d_{j+1}-1}a_k+d_{j+1}-d_{i-1}-1}{j-i+1}
	\end{array}
	\right).
\end{equation*}

By calculation, we can see that the maximal value of $T^{\lambda}_{A_{n-1}/P_I}$ is
\begin{align}
	M^{\lambda}_{A_{n-1}/P_I}=\max_{1\le i\le s}\big\{\sum\limits_{k=d_{i-1}+1}^{d_{i+1}-1}(a_k+1)\big\}.  
\end{align}

\textbf{II. The associated datum on $B_n/P_I$ and $C_n/P_I$}

\textbf{Case (a):} $I=\{d_1,\ldots,d_s\}$ and $d_s\ne n$.

By Lemma \ref{fund}, we see that ($\sum l_ie_i$ is abbreviated $(l_1,\ldots,l_n)$ in the following)
\begin{align*}
	\lambda+\rho&=(a_1+1)\lambda_1+\cdots+(a_n+1)\lambda_n\\
	&=\big(\sum_{i=1}^{n-1}(a_i+1)+\mathfrak{e}(a_n+1),\sum_{i=2}^{n-1}(a_i+1)+\mathfrak{e}(a_n+1),\ldots,a_{n-1}+1+\mathfrak{e}(a_n+1),\mathfrak{e}(a_n+1)\big),    
\end{align*}
and
\begin{align*}
	\sum_{i\in I}\lambda_{i}&=\lambda_{d_1}+\cdots+\lambda_{d_s} \\
	&=(\underbrace{s,\ldots,s}_\text{$d_1$~terms},\underbrace{s-1,\ldots,s-1}_\text{$(d_2-d_1)$~terms},\ldots,\underbrace{1,\ldots,1}_\text{$(d_s-d_{s-1})$~terms},\underbrace{0,\ldots,0}_\text{$(n-d_s)$~terms}).     
\end{align*}
Hence $\alpha\in \Phi^{+}_{G/P_I}$ is equivalent to saying that 
\begin{enumerate}
	\item[(1)] $\alpha=e_{d_i-u+1}-e_{d_j+v}~(1\le u\le d_i-d_{i-1},1\le v\le d_{j+1}-d_j)$, where $1\le i\le j\le s$, or
	\item[(2)] $\alpha=e_{d_{i-1}+u}+e_{d_j+v}~(1\le u\le d_i-d_{i-1},1\le v\le d_{j+1}-d_j)$, where $1\le i\le j\le s$, or
	\item[(3)] $\alpha=e_{d_{i-1}+u}+e_{d_{i-1}+v}~(1\le u\le v\le d_i-d_{i-1})$, where $1\le i\le s$.
\end{enumerate}
(Note that here we replace the root $e_t\in \Phi^{+}_B$ with $2e_t$ for the value of $\frac{(\lambda+\rho,e_t)}{(\sum_{i\in I}\lambda_{i},e_t)}$ is equal to $\frac{(\lambda+\rho,2e_t)}{(\sum_{i\in I}\lambda_{i},2e_t)}$.)

Fix integers $i,j~(1\le i\le j\le s)$. We separately calculate the value $\frac{(\lambda+\rho,\alpha)}{(\sum_{i\in I}\lambda_{i},\alpha)}$ for the above three cases.
\begin{align}\label{P}
	P^{ij}_{uv}:=\frac{(\lambda+\rho,e_{d_i-u+1}-e_{d_j+v})}{(\sum_{i\in I}\lambda_{i},e_{d_i-u+1}-e_{d_j+v})}=\frac{\sum\limits_{k=d_i-u+1}^{d_j+v-1}(a_k+1)}{j-i+1},
\end{align}
\begin{align}\label{Q}
	Q^{ij}_{uv}:=\frac{(\lambda+\rho,e_{d_{i-1}+u}+e_{d_j+v})}{(\sum_{i\in I}\lambda_{i},e_{d_{i-1}+u}-e_{d_j+v})}=\frac{\sum\limits_{k=d_{i-1}+u}^{n-1}(a_k+1)+\sum\limits_{k=d_j+v}^{n-1}(a_k+1)+2\mathfrak{e}(a_n+1)}{2s+1-(i+j)}, 
\end{align}
and
\begin{align}\label{R}
	R^{i}_{uv}:=\frac{(\lambda+\rho,e_{d_{i-1}+u}+e_{d_{i-1}+v})}{(\sum_{i\in I}\lambda_{i},e_{d_{i-1}+u}+e_{d_{i-1}+v})}=\frac{\sum\limits_{k=d_{i-1}+u}^{n-1}(a_k+1)+\sum\limits_{k=d_{i-1}+v}^{n-1}(a_k+1)+2\mathfrak{e}(a_n+1)}{2(s+1-i)}.    
\end{align}
There we conclude that $T^{\lambda}_{G/P_I}$ consists of matrices $(P^{ij})$, $(Q^{ij})~(1\le i\le j\le s)$ and $(R^{i})~(1\le i\le s)$. By calculation, it's not hard to get that
\begin{align}\label{seq}
	M^{\lambda}_{G/P_I}=\max\big\{\max_
	{1\le i\le s-1}\{\sum\limits_{k=d_{i-1}+1}^{d_{i+1}-1}(a_k+1)\}, \sum\limits_{k=d_{s-1}+1}^{n-1}(a_k+1)+
	\sum\limits_{k=d_s+1}^{n-1}(a_k+1)+2\mathfrak{e}(a_n+1)\big\}. 
\end{align}

\textbf{Case (b):} $I=\{d_1,\ldots,d_s\}$ and $d_s=n$.

In this case, 
\begin{align*}
	\sum_{i\in I}\lambda_{i}&=\lambda_{d_1}+\cdots+\lambda_{d_s} \\
	&=(\underbrace{s-1+\mathfrak{e},\ldots,s-1+\mathfrak{e}}_\text{$d_1$~terms},\underbrace{s-2+\mathfrak{e},\ldots,s-2+\mathfrak{e}}_\text{$(d_2-d_1)$~terms},\ldots,\underbrace{1+\mathfrak{e},\ldots,1+\mathfrak{e}}_\text{$(d_{s-1}-d_{s-2})$~terms},\underbrace{\mathfrak{e},\ldots,\mathfrak{e}}_\text{$(n-d_{s-1})$~terms}).     
\end{align*}
Thus $\alpha\in \Phi^{+}_{G/P_I}$ is equivalent to saying that 
\begin{enumerate}
	\item[(1)] $\alpha=e_{d_i-u+1}-e_{d_j+v}~(1\le u\le d_i-d_{i-1},1\le v\le d_{j+1}-d_j)$, where $1\le i\le j\le s-1$, or
	\item[(2)] $\alpha=e_{d_{i-1}+u}+e_{d_j+v}~(1\le u\le d_i-d_{i-1},1\le v\le d_{j+1}-d_j)$, where $1\le i\le j\le s-1$, or\\
	(notice that the above two cases don't happen when $s=1$.)
	\item[(3)] $\alpha=e_{d_{i-1}+u}+e_{d_{i-1}+v}~(1\le u\le v\le d_i-d_{i-1})$, where $1\le i\le s$.
\end{enumerate}

Fix integers $i,j~(1\le i\le j\le s-1)$. We have
\begin{align}
	\tilde{P}^{ij}_{uv}:=\frac{(\lambda+\rho,e_{d_i-u+1}-e_{d_j+v})}{(\sum_{i\in I}\lambda_{i},e_{d_i-u+1}-e_{d_j+v})}=\frac{\sum\limits_{k=d_i-u+1}^{d_j+v-1}(a_k+1)}{j-i+1},
\end{align}
\begin{align}
	\tilde{Q}^{ij}_{uv}:=\frac{(\lambda+\rho,e_{d_{i-1}+u}+e_{d_j+v})}{(\sum_{i\in I}\lambda_{i},e_{d_{i-1}+u}-e_{d_j+v})}=\frac{\sum\limits_{k=d_{i-1}+u}^{n-1}(a_k+1)+\sum\limits_{k=d_j+v}^{n-1}(a_k+1)+2\mathfrak{e}(a_n+1)}{2(s+\mathfrak{e})-(i+j+1)}, 
\end{align}
and fix integer $i~(1\le i\le s)$, we have
\begin{align}
	\tilde{R}^{i}_{uv}:=\frac{(\lambda+\rho,e_{d_{i-1}+u}+e_{d_{i-1}+v})}{(\sum_{i\in I}\lambda_{i},e_{d_{i-1}+u}+e_{d_{i-1}+v})}=\frac{\sum\limits_{k=d_{i-1}+u}^{n-1}(a_k+1)+\sum\limits_{k=d_{i-1}+v}^{n-1}(a_k+1)+2\mathfrak{e}(a_n+1)}{2(s+\mathfrak{e}-i)}.    
\end{align}
The matrices $(\tilde{P}^{ij})$, $(\tilde{Q}^{ij})~(1\le i\le j\le s-1)$ and $(\tilde{R}^{i})~(1\le i\le s)$ are the associate datum of $G/P_{d_1,\ldots,d_{s-1},n}$, where $G$ is of type $B$ or $C$. And 
\begin{align}
	M^{\lambda}_{G/P_{d_1,\ldots,d_{s-1},n}}=\max\big\{\max_
	{1\le i\le s-1}\{\sum\limits_{k=d_{i-1}+1}^{d_{i+1}-1}(a_k+1)\}, \sum\limits_{k=d_{s-1}+1}^{n-1}\frac{1}{\mathfrak{e}}(a_k+1)+
	(a_n+1)\big\}. 
\end{align}

\textbf{III. The associated datum on $D_n/P_I$ }

We need to consider the following four cases separately.

\textbf{Case (a):} $I=\{d_1,\ldots,d_s\}$ and $d_s\le n-2$.

By Lemma \ref{fund}, we have
\begin{align*}
	\lambda+\rho=&(a_1+1)\lambda_1+\cdots+(a_n+1)\lambda_n\\
	=&\big(\sum_{i=1}^{n-2}(a_i+1)+\frac{1}{2}(a_{n-1}+a_n+2),\sum_{i=2}^{n-2}(a_i+1)+\frac{1}{2}(a_{n-1}+a_n+2),\ldots,\\
	&\frac{1}{2}(a_{n-1}+a_n+2),\frac{1}{2}(a_n-a_{n-1})\big),    
\end{align*}
and
\begin{align*}
	\sum_{i\in I}\lambda_{i}&=\lambda_{d_1}+\cdots+\lambda_{d_s} \\
	&=(\underbrace{s,\ldots,s}_\text{$d_1$~terms},\underbrace{s-1,\ldots,s-1}_\text{$(d_2-d_1)$~terms},\ldots,\underbrace{1,\ldots,1}_\text{$(d_s-d_{s-1})$~terms},\underbrace{0,\ldots,0}_\text{$(n-d_s)$~terms}).     
\end{align*}
Hence $\alpha\in \Phi^{+}_{G/P_I}$ is equivalent to saying that 
\begin{enumerate}
	\item[(1)] $\alpha=e_{d_i-u+1}-e_{d_j+v}~(1\le u\le d_i-d_{i-1},1\le v\le d_{j+1}-d_j)$, where $1\le i\le j\le s$, or
	\item[(2)] $\alpha=e_{d_{i-1}+u}+e_{d_j+v}~(1\le u\le d_i-d_{i-1},1\le v\le d_{j+1}-d_j)$, where $1\le i\le j\le s$, or
	\item[(3)] $\alpha=e_{d_{i-1}+u}+e_{d_{i-1}+v}~(1\le u< v\le d_i-d_{i-1})$, where $1\le i\le s$.
\end{enumerate}
Fix integers $i,j~(1\le i\le j\le s)$. We separately calculate the value $\frac{(\lambda+\rho,\alpha)}{(\sum_{i\in I}\lambda_{i},\alpha)}$ for the above three cases.

\begin{align}
	P^{ij}_{uv}:=\frac{(\lambda+\rho,e_{d_i-u+1}-e_{d_j+v})}{(\sum_{i\in I}\lambda_{i},e_{d_i-u+1}-e_{d_j+v})}=\frac{\sum\limits_{k=d_i-u+1}^{d_j+v-1}(a_k+1)}{j-i+1},
\end{align}
\begin{align}
	Q^{ij}_{uv}:=\frac{(\lambda+\rho,e_{d_{i-1}+u}+e_{d_j+v})}{(\sum_{i\in I}\lambda_{i},e_{d_{i-1}+u}-e_{d_j+v})}=\frac{\sum\limits_{k=d_{i-1}+u}^{n-2}(a_k+1)+\sum\limits_{k=d_j+v}^{n}(a_k+1)}{2s+1-(i+j)}, 
\end{align}
and
\begin{align}
	R^{i}_{uv}:=\frac{(\lambda+\rho,e_{d_{i-1}+u}+e_{d_{i-1}+v})}{(\sum_{i\in I}\lambda_{i},e_{d_{i-1}+u}+e_{d_{i-1}+v})}=\frac{\sum\limits_{k=d_{i-1}+u}^{n-2}(a_k+1)+\sum\limits_{k=d_{i-1}+v}^{n}(a_k+1)}{2(s+1-i)}.    
\end{align}
$T^{\lambda}_{G/P_I}$ consists of matrices $(P^{ij})$, $(Q^{ij})~(1\le i\le j\le s)$ and $(R^{i})~(1\le i\le s)$. By calculaion, we have
\begin{align}
	M^{\lambda}_{G/P_I}=\max\big\{\max_
	{1\le i\le s-1}\{\sum\limits_{k=d_{i-1}+1}^{d_{i+1}-1}(a_k+1)\}, \sum\limits_{k=d_{s-1}+1}^{n-2}(a_k+1)+\sum\limits_{k=d_s+1}^{n}(a_k+1)\big\}. 
\end{align}

\textbf{Case (b):} $I=\{d_1,\ldots,d_s\}$ and $d_s=n-1$.

In this case, we calculate that
\begin{align*}
	\sum_{i\in I}\lambda_{i}&=\lambda_{d_1}+\cdots+\lambda_{d_s} \\
	&=(\underbrace{s-\frac{1}{2},\ldots,s-\frac{1}{2}}_\text{$d_1$~terms},\underbrace{s-\frac{3}{2},\ldots,s-\frac{3}{2}}_\text{$(d_2-d_1)$~terms},\ldots,\underbrace{\frac{1}{2},\ldots,\frac{1}{2}}_\text{$(d_s-d_{s-1})$~terms},\underbrace{-\frac{1}{2}}_\text{$1$~term}).     
\end{align*}
Hence $\alpha\in \Phi^{+}_{G/P_I}$ is equivalent to saying that 
\begin{enumerate}
	\item[(1)] $\alpha=e_{d_i-u+1}-e_{d_j+v}~(1\le u\le d_i-d_{i-1},1\le v\le d_{j+1}-d_j)$, where $1\le i\le j\le s$, or
	\item[(2)] $\alpha=e_{d_{i-1}+u}+e_{d_j+v}~(1\le u\le d_i-d_{i-1},1\le v\le d_{j+1}-d_j)$, where $1\le i\le j\le s$ and $i\ne s$, or\\
	(notice that this doesn't happen when $s=1$.)
	\item[(3)] $\alpha=e_{d_{i-1}+u}+e_{d_{i-1}+v}~(1\le u< v\le d_i-d_{i-1})$, where $1\le i\le s$.
\end{enumerate}
From the above argument, we know the associated datum $T^{\lambda}_{G/P_I}$ consists of the following matrices:
\begin{align}\label{p}
	\tilde{P}^{ij}_{uv}:=\frac{(\lambda+\rho,e_{d_i-u+1}-e_{d_j+v})}{(\sum_{i\in I}\lambda_{i},e_{d_i-u+1}-e_{d_j+v})}=\frac{\sum\limits_{k=d_i-u+1}^{d_j+v-1}(a_k+1)}{j-i+1},
\end{align}
\begin{align}\label{q}
	\tilde{Q}^{ij}_{uv}:=\frac{(\lambda+\rho,e_{d_{i-1}+u}+e_{d_j+v})}{(\sum_{i\in I}\lambda_{i},e_{d_{i-1}+u}-e_{d_j+v})}=\frac{\sum\limits_{k=d_{i-1}+u}^{n-2}(a_k+1)+\sum\limits_{k=d_j+v}^{n}(a_k+1)}{2s-(i+j)}, 
\end{align}
and
\begin{align}\label{r}
	\tilde{R}^{i}_{uv}:=\frac{(\lambda+\rho,e_{d_{i-1}+u}+e_{d_{i-1}+v})}{(\sum_{i\in I}\lambda_{i},e_{d_{i-1}+u}+e_{d_{i-1}+v})}=\frac{\sum\limits_{k=d_{i-1}+u}^{n-2}(a_k+1)+\sum\limits_{k=d_{i-1}+v}^{n}(a_k+1)}{2(s-i)+1}.    
\end{align}

By calculaion, we have
\begin{align}
	M^{\lambda}_{G/P_I}=\max\big\{&\max_
	{\substack{1\le i\le s\\i\ne s-1}}\{\sum\limits_{k=d_{i-1}+1}^{d_{i+1}-1}(a_k+1)\}, \sum\limits_{k=d_{s-2}+1}^{n-2}(a_k+1)+(a_n+1),\\
	&\sum\limits_{k=d_{s-1}+1}^{n-2}(a_k+1)+\sum\limits_{k=d_{s-1}+2}^{n}(a_k+1)\big\}. 
\end{align}
\\

When $s\ge 2$, we need to consider the remaining two cases.

\textbf{Case (c):} $I=\{d_1,\ldots,d_s\}$ and $d_s=n$ and $d_{s-1}\ne n-1$.

In this case, we have
\begin{align*}
	\sum_{i\in I}\lambda_{i}&=\lambda_{d_1}+\cdots+\lambda_{d_s} \\
	&=(\underbrace{s-\frac{1}{2},\ldots,s-\frac{1}{2}}_\text{$d_1$~terms},\underbrace{s-\frac{3}{2},\ldots,s-\frac{3}{2}}_\text{$(d_2-d_1)$~terms},\ldots,\underbrace{\frac{1}{2},\ldots,\frac{1}{2}}_\text{$(n-d_{s-1})$~terms}).      
\end{align*}
Hence $\alpha\in \Phi^{+}_{G/P_I}$ is equivalent to saying that 
\begin{enumerate}
	\item[(1)] $\alpha=e_{d_i-u+1}-e_{d_j+v}~(1\le u\le d_i-d_{i-1},1\le v\le d_{j+1}-d_j)$, where $1\le i\le j\le s-1$, or
	\item[(2)] $\alpha=e_{d_{i-1}+u}+e_{d_j+v}~(1\le u\le d_i-d_{i-1},1\le v\le d_{j+1}-d_j)$, where $1\le i\le j\le s-1$, or
	\item[(3)] $\alpha=e_{d_{i-1}+u}+e_{d_{i-1}+v}~(1\le u< v\le d_i-d_{i-1})$, where $1\le i\le s$.
\end{enumerate}
By calculation, we find that the associated datum $T^{\lambda}_{G/P_I}$ in Case (c) is the same as the datum in Case (b), i.e. 
$$T^{\lambda}_{G/P_I}=(\tilde{P}^{ij},\tilde{Q}^{ij}~(1\le i\le j\le s-1),\tilde{R}^{i}~(1\le i\le s))~(see (\ref{p}), (\ref{q}), (\ref{r})). $$
And
\begin{align}
	M^{\lambda}_{G/P_I}=\max\big\{\max_
	{1\le i\le s-1}\{\sum\limits_{k=d_{i-1}+1}^{d_{i+1}-1}(a_k+1)\}, \sum\limits_{k=d_{s-1}+1}^{n-2}(a_k+1)+\sum\limits_{k=d_{s-1}+2}^{n}(a_k+1)\big\}.    
\end{align}

\textbf{Case (d):} $I=\{d_1,\ldots,d_s\}$ and $d_s=n$ and $d_{s-1}=n-1$.

In this case, we have
\begin{align*}
	\sum_{i\in I}\lambda_{i}&=\lambda_{d_1}+\cdots+\lambda_{d_s} \\
	&=(\underbrace{s-1,\ldots,s-1}_\text{$d_1$~terms},\ldots,\underbrace{1,\ldots,1}_\text{$(d_{s-1}-d_{s-2})$~terms},\underbrace{0}_\text{$1$~term}).     
\end{align*}
Hence $\alpha\in \Phi^{+}_{G/P_I}$ is equivalent to saying that 
\begin{enumerate}
	\item[(1)] $\alpha=e_{d_i-u+1}-e_{d_j+v}~(1\le u\le d_i-d_{i-1},1\le v\le d_{j+1}-d_j)$, where $1\le i\le j\le s-1$, or
	\item[(2)] $\alpha=e_{d_{i-1}+u}+e_{d_j+v}~(1\le u\le d_i-d_{i-1},1\le v\le d_{j+1}-d_j)$, where $1\le i\le j\le s-1$, or
	\item[(3)] $\alpha=e_{d_{i-1}+u}+e_{d_{i-1}+v}~(1\le u< v\le d_i-d_{i-1})$, where $1\le i\le s-1$.
\end{enumerate}
From the above argument, we know the associated datum consists of the following matrices:
\begin{align}
	\hat{P}^{ij}_{uv}:=\frac{(\lambda+\rho,e_{d_i-u+1}-e_{d_j+v})}{(\sum_{i\in I}\lambda_{i},e_{d_i-u+1}-e_{d_j+v})}=\frac{\sum\limits_{k=d_i-u+1}^{d_j+v-1}(a_k+1)}{j-i+1},
\end{align}
\begin{align}
	\hat{Q}^{ij}_{uv}:=\frac{(\lambda+\rho,e_{d_{i-1}+u}+e_{d_j+v})}{(\sum_{i\in I}\lambda_{i},e_{d_{i-1}+u}-e_{d_j+v})}=\frac{\sum\limits_{k=d_{i-1}+u}^{n-2}(a_k+1)+\sum\limits_{k=d_j+v}^{n}(a_k+1)}{2s-1-(i+j)}, 
\end{align}
and
\begin{align}
	\hat{R}^{i}_{uv}:=\frac{(\lambda+\rho,e_{d_{i-1}+u}+e_{d_{i-1}+v})}{(\sum_{i\in I}\lambda_{i},e_{d_{i-1}+u}+e_{d_{i-1}+v})}=\frac{\sum\limits_{k=d_{i-1}+u}^{n-2}(a_k+1)+\sum\limits_{k=d_{i-1}+v}^{n}(a_k+1)}{2(s-i)}.
\end{align} 
By calculation, we get that
\begin{align}
	M^{\lambda}_{G/P_I}=\max\big\{\max_
	{1\le i\le s-1}\{\sum\limits_{k=d_{i-1}+1}^{d_{i+1}-1}(a_k+1)\}, \sum\limits_{k=d_{s-2}+1}^{n-2}+(a_n+1)\big\}. 
\end{align}

\subsubsection{The associated datum for exceptional type}
In this section, we explicitly describe the associated datum $T^{\lambda}_{G/P_I}$ on $G/P_I$, where $G$ is an exceptional Lie group of type $E_n~(n=6,7,8)$, $F_4$ or $G_2$. Recall that the positive roots of an exceptional simple Lie group are usually expressed as linear combinations of orthogonal bases (see \cite{carter2005lie} Section $8$). However, these expressions are so complicated that it is difficult to use them to calculate the associated datum. So, for the sake of convenience, we express these positive roots as combinations of simple roots. In \cite{Fang2023} Lemma 3.8, the authors have listed all the combinations in terms of the lexicographical order.
In this section, we write a positive root $\alpha$ as a combination of simple roots $\alpha_i~(1\le i\le n)$, 
$$\alpha=l_1\alpha_1+\cdots+l_n\alpha_n,~l_i\ge 0~(1\le i\le n).$$
From the relationship between the fundamental weights and the simple roots (see Section 2.1 (4)), it is not hard to see that
\[
\alpha\in \Phi^{+}_{G/P_I}~\text{if and only if}~
(l_{d_1},\ldots,l_{d_s})\ne(0,\ldots,0).
\]
Since $\lambda=a_1\lambda_1+\cdots+a_n\lambda_n$ and $I=\{d_1,\ldots,d_s\}\subset\{1,\ldots,n\}$, we compute
\begin{align*}
	(\lambda+\rho,\alpha)&=(\sum_{i=1}^{n}(a_i+1)\lambda_i,\sum_{i=1}^{n}l_i\alpha_i)
	=\sum_{i=1}^{n}(a_i+1).l_i(\lambda_i,\alpha_i)\\
	&=\frac{1}{2}\sum_{i=1}^{n}(a_i+1).l_i(\alpha_i,\alpha_i),    
\end{align*}
and
\begin{align*}
	(\sum_{i\in I}\lambda_{i},\alpha)&=(\sum_{i=1}^{s}\lambda_{d_i},\sum_{i=1}^{n}l_i\alpha_i)=\sum_{i=1}^{s}l_{d_i}(\lambda_{d_i},\alpha_{d_i})\\
	& =\frac{1}{2}\sum_{i=1}^{s}l_{d_i}(\alpha_{d_i},\alpha_{d_i}).
\end{align*}

\textbf{I. The associated datum on $E_n/P_I~(n=6,7,8)$}

Since all simple roots of $E_n~(n=6,7,8)$ are long roots, the Killing form $(\alpha_i,\alpha_i)=(\alpha_j,\alpha_j)$ for any $1\le i,j\le n$. Therefore, we have
\begin{align*}
	\frac{(\lambda+\rho,\alpha)}{(\sum_{i\in I}\lambda_{i},\alpha)}= \frac{\sum_{i=1}^{n}(a_i+1).l_i}{\sum_{i=1}^{s}l_{d_i}}.  
\end{align*}
It follows that the associated datum of $E_{\lambda}$ on $E_n/P_I$ is
\begin{align}
T^{\lambda}_{E_n/P_I}=\left\{\frac{\sum_{i=1}^{n}(a_i+1).l_i}{\sum_{i=1}^{s}l_{d_i}} \bigm|\alpha=\sum_{i=1}^{n}l_i\alpha_i~\text{and}~(l_{d_1},\ldots,l_{d_s})\ne(0,\ldots,0)\right\}.
\end{align}

\textbf{II. The associated datum on $F_4/P_I$}

From the Dynkin diagram of $F_4$, we observe that the simple roots $\alpha_1$, $\alpha_2$ of $F_4$ are long roots and $\alpha_3$, $\alpha_4$ are short roots. Since the nodes $2$ and $3$ are joined by two edges, 
$$(\alpha_1,\alpha_1)=(\alpha_2,\alpha_2)=2(\alpha_3,\alpha_3)=2(\alpha_4,\alpha_4).$$
There, we have
\begin{align*}
	\frac{(\lambda+\rho,\alpha)}{(\sum_{i\in I}\lambda_{i},\alpha)}= \frac{2\sum_{i=1}^{2}(a_i+1).l_i+\sum_{j=3}^{4}(a_j+1).l_j}{\sum_{i=1}^{s}t_{d_i}.l_{d_i}},  
\end{align*}
where 
\[
t_{d_i}=\left\{\begin{array}{cc}
	2,& \text{if}~ d_i=1 ~\text{or}~ 2,\\
	1,& \text{otherwise}.\\
\end{array}\right.  
\]
Moreover,
\begin{align}
T^{\lambda}_{F_4/P_I}=\left\{\frac{2\sum_{i=1}^{2}(a_i+1).l_i+\sum_{j=3}^{4}(a_j+1).l_j}{\sum_{i=1}^{s}t_{d_i}.l_{d_i}} \bigm|\alpha=\sum_{i=1}^{4}l_i\alpha_i~\text{and}~(l_{d_1},\ldots,l_{d_s})\ne(0,\ldots,0)\right\}.
\end{align}

\textbf{III. The associated datum on $G_2/P_I$}

Note that the simple roots $\alpha_1$ and $\alpha_2$
are short and long roots of $G_2$ respectively. Since
$(\alpha_2,\alpha_2)=3(\alpha_1,\alpha_1)$, we have
\begin{align*}
	\frac{(\lambda+\rho,\alpha)}{(\sum_{i\in I}\lambda_{i},\alpha)}= \frac{l_1(a_1+1)+3l_2(a_2+1)}{\sum_{i=1}^{s}t'_{d_i}.l_{d_i}},  
\end{align*}
where $t'_{d_i}=1$ if $d_i=1$ and otherwise $t'_{d_i}=3$. Moreover,
\begin{align}\label{g2}
T^{\lambda}_{G_2/P_I}=\left\{\frac{l_1(a_1+1)+3l_2(a_2+1)}{\sum_{i=1}^{s}t'_{d_i}.l_{d_i}} \bigm|\alpha=\sum_{i=1}^{2}l_i\alpha_i~\text{and}~(l_{d_1},\ldots,l_{d_s})\ne(0,\ldots,0)\right\}.
\end{align}
\begin{ex}
Let's consider the simple Lie group $G_2$. Since the positive roots of $G_2$ are $\alpha_1,~\alpha_2,~\alpha_1+\alpha_2,~2\alpha_1+\alpha_2,~3\alpha_1+\alpha_2$, and $3\alpha_1+2\alpha_2$, the associate datum of $L_{\lambda}~(\lambda=a_1\lambda_1+a_2\lambda_2)$ on $G_2/P_{1,2}$ is of the form (see the sequence (\ref{g2}))
$$\{a_1+1,a_2+1,\frac{a_1+3a_2}{4}+1,\frac{2a_1+3a_2}{5}+1,\frac{a_1+a_2}{2}+1,\frac{a_1+2a_2}{3}+1\}.$$
Without loss of generality, we assume $\min\{a_1,a_2\}=0$. If $a_1=0$, then  $T^{\lambda}_{G_2/P_{1,2}}=\{1, a_2+1,\frac{3a_2}{4}+1,\frac{3a_2}{5}+1,\frac{a_2}{2}+1,\frac{2a_2}{3}+1\}$. When $a_2=0$, it's obvious that $L_{\lambda}$ is an ACM bundle. When $a_2>0$, the statement $L_{\lambda}$ is an ACM bundle forces to have $2\in T^{\lambda}_{G_2/P_{1,2}}$, then the only choices of $a_2$ are $1$ and $2$. It's easy to check $L_{\lambda}$ is an ACM bundle when $a_1=0$ and $a_2=0,1,2$. 
Arguing in the same way, if $a_2=0$, then $L_{\lambda}$ is an ACM bundle if and only if $a_1=0,1,2$. Hence the irreducible homogeneous ACM bundles on $G_2/P_{1,2}$ can only be (up to twist) line bundles $L_{a_1\lambda_1+a_2\lambda_2}$, where $0\le a_1,a_2\le 2$. 
\end{ex}
The above example together with Example 3.12 in \cite{Fang2023} classifies the irreducible homogeneous ACM bundles on all rational homogeneous spaces $G_2/P$.  
 \section{Applications}
 In this section, we always assume that $G$ is a classical simple Lie group.  
 We use Theorem \ref{thm} to judge that some special homogeneous vector bundles on $X=G/P_{d_1,\ldots,d_s}$ with respect to the minimal ample class are ACM bundles. (For the case where $G$ is of type $A$, these assertions have already appeared in \cite{Fang2022}, and will not be repeated here.) We first recall that the minimum value of the associated datum of $E_{\lambda}$ on $X$ is $\min_{1\le i\le s}\{a_{d_i}\}+1$ (see Remark \ref{min}). Denote by
 \[
 \mu=\min_{1\le i\le s}\{a_{d_i}\}+1~\text{and}~\nu=\max_{\substack{1\le i\le s\\a_{d_i}+1=\mu}}\{d_{i+1}-d_{i-1}\}-1.
 \]
 
 \begin{prop}\label{prop1}
 	Let $X=B_n/P_{d_1,\ldots,d_s}$ with $d_s\le n-1$. Fix an integer $m~(1\le m\le s+1)$. Let $E$ be an irreducible homogeneous bundle on $X$ with highest weight 
 	$$\sum_{\substack{1\le i\le s\\i\ne m}}a_{d_i}\lambda_{d_i}+\sum\limits_{j=d_{m-1}+1}^{d_m}a_j\lambda_{j}.$$ 
 	If $|a_{d_i}-a_{d_j}|\le \nu$ for any $1\le i,j\le s$ and for any $d_{m-1}+1\le j\le d_m-1$,
 	\begin{align}\label{condition}
 		\left\{
 		\begin{array}{lll}
 			0\le a_j\le d_2-d_1-1, & \text{when}~m=1;\\
 			0\le a_j\le \min\{d_{m+1}-d_m,d_{m-1}-d_{m-2}\}-1,& \text{when}~2\le m\le s;\\
 			0\le a_j\le d_s-d_{s-1}-1~\text{and}~0\le a_n\le 2(d_s-d_{s-1})-1, & \text{when}~m=s+1,\\
 		\end{array}
 		\right.
 	\end{align}
 	then $E$ is an ACM bundle on $X$. 
 \end{prop}
 \begin{proof}
 	Firstly, for any $i~(1\le i\le s)$, we denote by 
 	\begin{align*}
 		u_i= a_{d_i}+1 
 	\end{align*}
 	and 
 	\begin{align*}
 		v_i=  \left\{
 		\begin{array}{ll}
 			\sum\limits_{k=d_{i-1}+1}^{d_{i+1}-1}(a_k+1), & \text{if}~1\le i\le s-1,\\
 			\sum\limits_{k=d_{s-1}+1}^{n-1}(a_k+1)+
 			\sum\limits_{k=d_s+1}^{n}(a_k+1), & \text{if}~i=s.\\
 		\end{array}
 		\right.
 	\end{align*}
 	By Remark \ref{min} and  the sequence (\ref{seq}), we have
 	$$\min_{1\le i\le s}u_i=\mu=m^{\lambda}_{X},~\max_{1\le i\le s}v_i=M^{\lambda}_{X}.$$
 	
 	Let $p~(1\le p\le s)$ be an integer such that $u_p=\mu=m^{\lambda}_{X}$ and $d_{p+1}-d_{p-1}-1=\nu$. Let $q~(1\le q\le s)$
 	be an integer with $v_q=M^{\lambda}_{X}$. According to the first hypothesis, we can obtain
 	\begin{align*}
 		u_q=a_{d_q}+1\le a_{d_p}+d_{p+1}-d_{p-1}-1 +1\le v_p+1.
 	\end{align*}
 	In addition, due to Proposition 3.9 of \cite{Costa2016} 
 	and Corollary 3.15 of \cite{Du2023}, one can find that the assumption (\ref{condition}) ensures that for any integer $i~(1\le i\le s)$ and any integer $l\in [u_i,v_i]$, $n_l\ge 1$. In particular, for any $l\in [u_p,v_p]$ and any $l\in [u_q,v_q]$, $n_l\ge 1$. 
 	
 	From the previous analysis, we already know that $u_q\le v_p+1$, so for any $l\in [u_p,v_q]$, $n_l\ge 1$. Since $u_p=m^{\lambda}_{X}$ and $v_q=M^{\lambda}_{X}$, we obtain that $E$ is an ACM bundle by Theorem \ref{thm}.
 \end{proof}
 
 With the very same proof, the above proposition  can be extended to homogeneous spaces $X=G/P$ with $G$ simple of type $C_n$ or $D_n$.
 \begin{prop}\label{prop2}
 	Let $X=C_n/P_{d_1,\ldots,d_s}$. Fix an integer $m~(1\le m\le s+1)$. Let $E$ be an irreducible homogeneous bundle on $X$ with highest weight 
 	$$\sum_{\substack{1\le i\le s\\i\ne m}}a_{d_i}\lambda_{d_i}+\sum\limits_{j=d_{m-1}+1}^{d_m}a_j\lambda_{j}.$$ 
 	If $|a_{d_i}-a_{d_j}|\le \nu$ for any $1\le i,j\le s$ and for any $d_{m-1}+1\le j\le d_m-1$,
 	\begin{align}
 		\left\{
 		\begin{array}{lll}
 			0\le a_j\le d_2-d_1-1, & \text{when}~m=1;\\
 			0\le a_j\le \min\{d_{m+1}-d_m,d_{m-1}-d_{m-2}\}-1,& \text{when}~2\le m\le s;\\
 			0\le a_j\le d_s-d_{s-1}-1~\text{and}~0\le a_n\le d_s-d_{s-1}-1, & \text{when}~m=s+1,\\
 		\end{array}
 		\right.
 	\end{align}
 	then $E$ is an ACM bundle on $X$. 
 \end{prop}
 \begin{proof}
 	\cite{Costa2016} Proposition 3.9 and Corollary 3.17 in \cite{Du2023} play important roles in the proof. Our proof applies almost verbatim the proof of Proposition \ref{prop1}.   
 \end{proof}
 
 \begin{prop}\label{prop3}
 	Let $X=D_n/P_{d_1,\ldots,d_s}$ with $d_s\le n-2$. Fix an integer $m~(1\le m\le s+1)$. Let $E$ be an irreducible homogeneous bundle on $X$ with highest weight 
 	$$\sum_{\substack{1\le i\le s\\i\ne m}}a_{d_i}\lambda_{d_i}+\sum\limits_{j=d_{m-1}+1}^{d_m}a_j\lambda_{j}.$$ 
 	If $|a_{d_i}-a_{d_j}|\le \nu$ for any $1\le i,j\le s$ and for any $d_{m-1}+1\le j\le d_m-1$,
 	\begin{align}
 		\left\{
 		\begin{array}{lll}
 			0\le a_j\le d_2-d_1-1, & \text{when}~m=1;\\
 			0\le a_j\le \min\{d_{m+1}-d_m,d_{m-1}-d_{m-2}\}-1,& \text{when}~2\le m\le s;\\
 			0\le a_j\le d_s-d_{s-1}-1~(d_s+1\le j\le n-2)~\text{and} & \text{when}~m=s+1\\
 			0\le a_{n-1}\le d_s-d_{s-1}-1,0\le a_n-a_{n-1}\le 2(d_s-d_{s-1})-1~\text{or},\\
 			0\le a_n\le d_s-d_{s-1}-1,0\le a_{n-1}-a_n\le 2(d_s-d_{s-1})-1,\\
 		\end{array}
 		\right.
 	\end{align}
 	then $E$ is an ACM bundle on $X$. 
 \end{prop}
 \begin{proof}
 	\cite{Costa2016} Proposition 3.9 and Corollary 3.18 in \cite{Du2023} play important roles in the proof. Our proof applies almost verbatim the proof of Proposition \ref{prop1}.   
 \end{proof}
 
 
 \begin{cor}\label{cor} 
 	The universal bundles $H_{d_m}~(1\le m\le s)$ are ACM bundles on $X$. Moreover, if $d_s-d_{s-1}\ge 2$, then all $H_{d_m}^{\bot}~(1\le m\le s)$ are also ACM bundles.
 \end{cor}
 \begin{proof}
 	By Example \ref{ex}, we know that $H_{d_1}$ is an irreducible homogeneous bundle with highest weight $\lambda_{d_1-1}-\lambda_{d_1}$. If $d_2-d_1\ge 2$, then $H_{d_1}$ is an ACM bundle by Propositions \ref{prop1}, \ref{prop2}, \ref{prop3}. If $d_2-d_1=1$, we can still draw this conclusion by analyzing the associated datum of $H_{d_1}$ in detail. Similarly, it can be verified that for any fixed integer $m~(2\le m\le s)$, $H_{d_m}/H_{d_{m-1}}$ and its dual $(H_{d_m}/H_{d_{m-1}})^{\ast}$ are also ACM bundles. Since the extension of ACM bundles is ACM, we get that all $H_{d_m}$ are ACM bundles.
 	
 	Moreover, if $d_s-d_{s-1}\ge 2$, then based on the highest weight of $H_{d_s}^{\bot}/H_{d_s}$ (see Example \ref{ex}), we can deduce that
 	 $H_{d_s}^{\bot}/H_{d_s}$ is an ACM bundle due to Propositions \ref{prop1}, \ref{prop2}, \ref{prop3}. It follows that $H_{d_s}^{\bot}$ is also an ACM bundle. Since $H_{d_{m-1}}^{\bot}/H_{d_m}^{\bot}\cong (H_{d_m}/H_{d_{m-1}})^{\ast}$ are ACM bundles, we get that all $H_{d_m}^{\bot}$ are ACM bundles.
 \end{proof}	
 \begin{rem}
 	When $d_s-d_{s-1}=1$, $H_{d_s}^{\bot}/H_{d_s}$ may not be an ACM bundle. For example, when $G=B_n~(n\ge 3)$, $s=1$ and $d_1=1$, $X=B_n/P_1$ is a smooth quadric $\textbf{Q}_{2n-1}$. Theorem 3.5 in \cite{ottaviani1989} states that ACM bundles on $X$ can only be, up to twist and isomorphism, a direct sum of line bundles and the spinor bundle, whose highest weight is $\lambda_n$. However  $H_{d_1}^{\bot}/H_{d_1}=H_1^{\bot}/H_1$ is an irreducible homogeneous bundle on $\textbf{Q}_{2n-1}$ with highest weight $\lambda_2-\lambda_1$, hence it is not an ACM bundle.
 \end{rem}
 
If $X$ is a rational homogeneous space of Picard rank $1$, then by the Serre duality and the Kodaira vanishing theorem, it's easy to check that all line bundles on $X$ are ACM bundles. But for rational homogeneous spaces of Picard rank $>1$, we find that line bundles on them are not necessarily ACM bundles. However, using the associated datum of homogeneous bundles, we can say slightly more about line bundles on rational homogeneous spaces $G/P_{d_1,d_2}$. For convenience, we denote
 $$L_1:=L_{\lambda_{d_1}}~\text{and}~L_2:=L_{\lambda_{d_2}},$$
 where $L_{\lambda_{d_1}}$ and  $L_{\lambda_{d_2}}$ are the generators of $Pic(G/P_{d_1,d_2})$. Any line bundle on $G/P_{d_1,d_2}$ can be written as 
 $a_1L_1\otimes a_2L_2$. Recall that in \cite{Fang2022}, the author gave a necessary and sufficient condition for a line bundle on the usual flag variety $F(d_1,d_2,n):=A_{n-1}/P_{d_1,d_2}$ to be an ACM bundle. 
 \begin{prop}(\cite{Fang2022} Proposition 4.13)
 	Let $L=a_1L_1\otimes a_2L_2$ be a line bundle on the flag variety $F(d_1,d_2,n)$. Then $L$ is an ACM bundle if and only if $0\le a_2-a_1\le \min\{n,d_1+d_2\}-1$ or $0<a_1-a_2\le \min\{n,2n-d_1-d_2\}-1$. 
 \end{prop}
 
 Next, we will give the necessary and sufficient conditions under which the line bundles on $G/P_{d_1,d_2}$ with $G$ simple of types $B$, $C$ and $D$, are ACM bundles. 
 
 \begin{prop}
 	Let $L=a_1L_1\otimes a_2L_2$ be a line bundle on  $B_n/P_{d_1,d_2}$. 
 	\begin{enumerate}
 		\item[(a)] Suppose $d_2\le n-1$. Then $L$ is an ACM bundle if and only if
 		\begin{align*}
 			&0\le a_2-a_1\le \min\{d_1+d_2,2n-d_1\}-1~\text{or}\\
 			&0<a_1-a_2\le \min\{2n-d_2,4n-2(d_1+d_2)\}-1.  
 		\end{align*}
 		\item[(b)] Suppose $d_2=n$. Then $L$ is an ACM bundle if and only if
 		\begin{align*}
 			&0\le a_2-a_1\le \min\{2d_1+n,2n-d_1\}-1~\text{or}\\
 			&0<a_1-a_2\le \min\{3(n-d_1),2n\}-1.     \end{align*}  
 	\end{enumerate}
 \end{prop}
 
 \begin{proof}
 	First of all, we note that the weight of $L$ is $a_1\lambda_{d_1}+a_2\lambda_{d_2}$. For case (a), suppose $a_1\le a_2$, then we can assume without loss of generality that $a_1=0$. In this case, the associated datum of $L$ consists of the following eight matrices: (see Section 3.2.1 II Case (a) (\ref{P}), (\ref{Q}), (\ref{R}))
 	
 	\begin{align*}
 		\begin{split}&P^{11}_{uv}=u+v-1,~Q^{11}_{uv}=\frac{2a_2+2n+1-d_1-(u+v)}{3}~(1\le u\le d_1,1\le v \le d_2-d_1),\\&P^{22}_{uv}=a_2+u+v-1,~ Q^{22}_{uv}=a_2+2n+1-(d_1+d_2)-(u+v)~(1\le u\le d_2-d_1,1\le v \le n-d_2),\\&P^{12}_{uv}=\frac{a_2+d_2-d_1+u+v-1}{2},~Q^{12}_{uv}=\frac{a_2+2n+1-d_2-(u+v)}{2}~(1\le u\le d_1,1\le v \le n-d_2),\end{split}\end{align*}
 	and
 	\begin{align*}
 		&R^{1}_{uv}=\frac{2a_2+2n+1-(u+v)}{4}~(1\le u\le v\le d_1),\\
 		&R^{2}_{uv}=\frac{2a_2+2(n-d_1)+1-(u+v)}{2}~(1\le u\le v\le d_2-d_1).   
 	\end{align*}
 	By further analysis, we can see that 
 	\begin{enumerate}
 		\item[\romannumeral1.] any integer $1\le l\le d_2-1$ appears as an entry of $P^{11}$;
 		\item[\romannumeral2.] any integer $a_2+1\le l\le a_2+2n-1-(d_1+d_2)$ appears as an entry of $P^{22}$, $Q^{22}$ or $R^{2}$;
 		\item[\romannumeral3.] any integer $\frac{a_2+d_2+1-d_1}{2}\le l\le \frac{a_2+2n-d_2-1}{2}$ appears as an entry of $P^{12}$, $Q^{12}$ or $R^{1}$;
 		\item[\romannumeral4.] any integer $\frac{2a_2+2n+1-(d_1+d_2)}{3}\le l\le \frac{2a_2+2n-d_1-1}{3}$ appears as an entry of $Q^{11}$.
 	\end{enumerate}
 	
 	If $a_2\le d_2-1$, then $L$ is an ACM bundle by Proposition \ref{prop1}. 
 	
 	If $a_2>d_2-1$, then $a_2+1>d_2$ and $\frac{a_2+d_2+1-d_1}{2}<\frac{2a_2+2n+1-(d_1+d_2)}{3}$. Combining Theorem \ref{thm} and the above statements \romannumeral1, \romannumeral2, we get that $L$ is an ACM bundle if and only if all integer values between $d_2$ and $a_2$ appear as entries in the associated datum of $L$. 
 	Since $\frac{a_2+d_2+1-d_1}{2}<\frac{2a_2+2n+1-(d_1+d_2)}{3}$, this forces to have
 	$$\frac{a_2+d_2+1-d_1}{2}\le d_2,~\text{i.e.}~a_2\le d_2+d_1-1,$$
 	otherwise the value $d_2$ would not appear as an entry in the associated datum of $L$. Arguing in the same way we must have $a_2\le 2n-d_1-1$, otherwise the value $a_2$ would not appear as an entry in the associated datum of $L$. On the contrary, if $a_2\le 2n-d_1-1$, then it's easy to check that
 	$$\frac{2a_2+2n-d_1-1}{3}\ge a_2, ~\text{and}~\frac{2a_2+2n+1-(d_1+d_2)}{3}\le\lfloor \frac{a_2+2n-d_2-1}{2} \rfloor+1.$$
 	By the statements \romannumeral3, \romannumeral4, we get that all integer values between $d_2$ and $a_2$ appear as entries in the associated datum of $L$, hence $L$ is an ACM bundle. 
 	
 	Suppose $a_1>a_2$, then we can assume that $a_2=0$. After analyzing the associated datum of $L$ in this case, we can see that
 	\begin{enumerate}
 		\item[I.] any integer $a_1+1\le l\le a_1+d_2-1$ appears as an entry of $P^{11}$;
 		\item[II.] any integer $1\le l\le 2n-(d_1+d_2+1)$ appears as an entry of $P^{22}$, $Q^{22}$ or $R^{2}$;
 		\item[III.] any integer $\frac{a_1+d_2+1-d_1}{2}\le l\le \frac{a_1+2n-d_2-1}{2}$ appears as an entry of $P^{12}$, $Q^{12}$ or $R^{1}$;
 		\item[IV.] any integer $\frac{a_1+2n+1-(d_1+d_2)}{3}\le l\le \frac{a_1+2n-d_1-1}{3}$ appears as an entry of $Q^{11}$.
 	\end{enumerate}
 	
 	If $a_1\le 2n-(d_1+d_2+1)$, then $a_1+1\le 2n-(d_1+d_2)$. By the above statements I, II and Theorem \ref{thm}, we know immediately that $L$ is an ACM bundle. 
 	
 	If $a_1>2n-(d_1+d_2+1)$, then $a_1+1>2n-(d_1+d_2)$ and $\frac{a_1+2n-d_2-1}{2}>\frac{a_1+2n-d_1-1}{3}$. Combining Theorem \ref{thm} and the above statements I, II, we get that $L$ is an ACM bundle if and only if all integer values between $2n-(d_1+d_2)$ and $a_1$ appear as entries in the associated datum of $L$. 
 	Since $\frac{a_1+2n-d_2-1}{2}>\frac{a_1+2n-d_1-1}{3}$, this forces to have
 	$$\frac{a_1+2n-d_2-1}{2}\ge a_1,~\text{i.e.}~a_1\le 2n-d_2-1,$$
 	otherwise the value $a_1$ would not appear as an entry in the associated datum of $L$. Arguing in the same way we must have $a_1\le 4n-2(d_1+d_2)-1$, otherwise the value $2n-(d_1+d_2)$ would not appear as an entry in the associated datum of $L$. On the contrary, if $a_1\le 4n-2(d_1+d_2)-1$, then 
 	$$\frac{a_1+2n+1-(d_1+d_2)}{3}\ge 2n-(d_1+d_2), ~\text{and}~\frac{a_1+d_2+1-d_1}{2}\le\lfloor \frac{a_1+2n-d_1-1}{3} \rfloor+1.$$
 	By the statements III and IV, we get that all integer values between $2n-(d_1+d_2)$ and $a_1$ appear as entries in the associated datum of $L$. This implies that $L$ is an ACM bundle and concludes the proof. The proof of case (b) is similar, and so we omit it.
 \end{proof}
 
 With the similar proof, we can get the following results.
 \begin{prop}
 	Let $L=a_1L_1\otimes a_2L_2$ be a line bundle on  $C_n/P_{d_1,d_2}$. 
 	\begin{enumerate}
 		\item[(a)] Suppose $d_2\le n-1$. Then $L$ is an ACM bundle if and only if
 		\begin{align*}
 			&0\le a_2-a_1\le \min\{d_1+d_2,2n-(d_1-1)\}-1~\text{or}\\
 			&0<a_1-a_2\le \min\{2n-(d_2-1),4n-2(d_1+d_2-1)\}-1.  
 		\end{align*}
 		\item[(b)] Suppose $d_2=n$. Then $L$ is an ACM bundle if and only if
 		\begin{align*}
 			&0\le a_2-a_1\le \min\{d_1+n-1,2n-d_1\}~\text{or}\\
 			&0<a_1-a_2\le \min\{2(n-d_1)+1,n\}.  \end{align*}
 	\end{enumerate}
 \end{prop}
 
 \begin{prop}
 	Let $L=a_1L_1\otimes a_2L_2$ be a line bundle on  $D_n/P_{d_1,d_2}$. 
 	\begin{enumerate}
 		\item[(a)] Suppose $d_2\le n-2$. Then $L$ is an ACM bundle if and only if
 		\begin{align*}
 			&0\le a_2-a_1\le \min\{d_1+d_2,2n-(d_1+1)\}-1~\text{or}\\
 			&0<a_1-a_2\le \min\{2n-(d_2+1),4n-2(d_1+d_2+1)\}-1.  
 		\end{align*}
 		\item[(b)] Suppose $d_2=n-1$ or $d_2=n,~d_1\le n-2$. Then $L$ is an ACM bundle if and only if
 		\begin{align*}
 			&0\le a_2-a_1\le \min\{2d_1+n,2n-d_1-1\}-1~\text{or}\\
 			&0<a_1-a_2\le \min\{3(n-d_1-1),2(n-1)\}-1.   
 		\end{align*}
 		\item[(c)] Suppose $d_2=n$ and $d_1=n-1$. Then $L$ is an ACM bundle if and only if
 		\begin{align*}
 			0\le |a_1-a_2| \le 2n-3.
 		\end{align*}
 	\end{enumerate}
 \end{prop}
\bibliography{ref}

\end{document}